\documentclass[a4paper,11pt]{article}
\usepackage{amsfonts}
\usepackage{latexsym}
\usepackage{amsmath}
\usepackage{amssymb}
\usepackage{amssymb}
\usepackage{slashed}
\usepackage{upgreek}
\usepackage{mathrsfs}
\usepackage{bbm}

\usepackage{amsthm}

\theoremstyle{remark}

\newtheorem*{rmk*}{Remark}

\theoremstyle{remark}
\newtheorem*{remark}{Remark}

\usepackage[utf8]{inputenc}
\usepackage[english]{babel}

\theoremstyle{definition}

\newtheorem{theorem}{Theorem}[section]
\newtheorem{corollary}{Corollary}[section]
\newtheorem{lemma}[theorem]{Lemma}

\newtheorem{prop}{{\bf Proposition}}[section]

\allowdisplaybreaks

\usepackage{relsize}
\usepackage{graphicx}
\usepackage{comment}

\renewcommand{\thefootnote}{\fnsymbol{footnote}}

\def\appendix#1{\addtocounter{section}{1}\setcounter{equation}{0}
\renewcommand{\thesection}{\Alph{section}}
\section*{Appendix \thesection\protect\indent \parbox[t]{11.15cm}{#1}}
\addcontentsline{toc}{section}{Appendix \thesection\ \ \ #1}}

\font\mybb=msbm10 at 11pt

\def\bb#1{\hbox{\mybb#1}}

\def\bZ {\bb{Z}}

\def\ma{\mathrm{a}}
\def\mb{\mathrm{b}}

\def\asd{\mathrm{asd}}
\def\HE{\mathrm{HE}}

\def\rsq {]\kern -2.0pt]}
\def\lsq {[\kern -2.0pt[}

\def\h{\widehat}
\def\be{\begin{equation}}
\def\ee{\end{equation}}

\def\ii{\iota}
\def\Adj{\mathrm{Ad}}

\def\rsq {]\kern -2.0pt]}
\def\lsq {[\kern -2.0pt[}

\newcommand{\bea}{\begin{eqnarray}}
\newcommand{\eea}{\end{eqnarray}}

\usepackage[usenames,dvipsnames,svgnames,table]{xcolor}	% AH -- REMOVE BEFORE UPLOAD
\usepackage[normalem]{ulem}				% AH -- REMOVE BEFORE UPLOAD
\usepackage{microtype}
\usepackage[colorlinks, citecolor=blue, linkcolor=blue, urlcolor=Maroon, filecolor=Maroon, linktocpage=true]{hyperref} % AH

\usepackage{chngcntr}
\counterwithout{equation}{section}

\begin{document}

\begin{center}
%\today
\vspace*{-1.0cm}
\begin{flushright}
%\normalsize{\texttt{ZMP-HH/17-24}}\\
\end{flushright}
%\hfill hep-th/yymmnnn \\
%\hfill UB-ECM-PF-06-43 \\
%\hfill DMUS--MP--13/06 \\

%\vspace{2.0cm} {\Large \bf A Resolution of the Lightcone Singularity and Vacuum States in Quantum Field Theory  } \\[.2cm]

\vspace{2.0cm} {\Large \bf Geometry of the moduli space of Hermitian-Einstein connections on  manifolds with a dilaton } \\[.2cm]

\vskip 2cm
 Georgios  Papadopoulos
\\
\vskip .6cm

%\begin{small}
%${}^1$ \textit{Department of Mathematics, King's College London
%\\
%Strand, London WC2R 2LS, UK}\\
%\texttt{sebastian.lautz@kcl.ac.uk}\\
%\end{small}
%\vskip0.5cm

\begin{small}
\textit{Department of Mathematics
\\
King's College London
\\
Strand
\\
 London WC2R 2LS, UK}\\
\texttt{george.papadopoulos@kcl.ac.uk}
\end{small}
\\*[.6cm]

\end{center}

\vskip 2.5 cm

\begin{abstract}
\noindent

We demonstrate that the moduli space of Hermitian-Einstein connections $\mathscr{M}^*_\HE(M^{2n})$ of vector bundles over compact non-Gauduchon Hermitian manifolds  $(M^{2n}, g, \omega)$ that  exhibit a dilaton field $\Phi$ admit a strong K\"ahler with torsion structure provided a certain condition is imposed on  their Lee form $\theta$ and the dilaton. We find that the geometries that satisfy this condition include those that  solve the string field equations or equivalently the  (generalised) gradient flow soliton type of equations.  In addition, we demonstrate that if the underlying manifold $(M^{2n}, g, \omega)$ admits a holomorphic and Killing vector field $X$ that leaves $\Phi$ also invariant, then the moduli spaces $\mathscr{M}^*_\HE(M^{2n})$  admits an induced holomorphic and Killing vector field $\alpha_X$. Furthermore, if $X$ is covariantly constant with respect to the compatible connection $\h\nabla$ with torsion a 3-form  on $(M^{2n}, g, \omega)$, then $\alpha_X$ is also covariantly constant with respect to the compatible connection $\h{\mathcal {D}}$  with torsion a 3-form on $\mathscr{M}^*_\HE(M^{2n})$ provided that $K^\flat\wedge X^\flat$ is a $(1,1)$-form with $K^\flat=\theta+2d\Phi$ and $\Phi$ is invariant under both $X$ and $IX$, where  $I$ is the complex structure of $M^{2n}$.

\end{abstract}

%\vskip 1cm

%{\small Keywords: spacetime geometry; black holes; special Lorentzian structures; G-structures}

\newpage

\renewcommand{\thefootnote}{\arabic{footnote}}
%\tableofcontents

%%%%%%%%%%%%%%%%%%%%%%%%%%%%%%%%%%%%%%%%%%%%%%%%%%%%%%%%%%%%%%%%%%%%%%%%%%
%\setcounter{section}{0}\setcounter{equation}{0}
%%%%%%%%%%%%%%%%%%%%%%%%%%%%%%%%%%%%%%%%%%%%%%%%%%%%%%%%%%%%%%%%%%%%%%%%%%

%\section{Introduction}
%%\counterwithout{equation}{chapter}

\numberwithin{equation}{section}

\section{Introduction}\label{intro}

\underline{{\it Prologue}:} Hermitian-Einstein connections \cite{kobayashi} of vector bundles $E$ over a manifold\footnote{The manifold $M^{2n}$ has metric $g$ and Hermitian form $\omega$, $\omega_{ij}=g_{ik} I^k{}_j$, where $I$ is the complex structure.} $(M^{2n}, g, \omega)$ and their moduli spaces $\mathscr{M}_\HE(M^{2n})$  have been extensively studied, first as  part of the theory of instantons on 4-dimensional manifolds, $n=2$, and later in higher than four dimensions $n>2$ \cite{Donaldson, Uhlenbeck, LiYau, Li, lubke}. Amongst their applications are those in physics, for example  they are used in  heterotic string supersymmetric compactifications \cite{Strominger, Hullheterotic}, see also \cite{mcorist1, mcorist2, mcorist3},  and more recently in AdS/CFT duality of Maldacena \cite{maldacena}.  In the latter case, Witten \cite{witten} used the geometry of the moduli space of instantons on the hyper-K\"ahler with torsion (HKT) manifold $S^3\times S^1$   to provide evidence for an   AdS$_3$/CFT$_2$ duality -- this  states that IIB strings on AdS$_3\times S^3\times S^3\times S^1$ are dual to the 2-dimensional sigma model with target space the smooth part of the moduli space $\mathscr{M}^*_\asd(S^3\times S^1)$  of instantons on $S^3\times S^1$ \cite{maldacena, boonstra, elitzur, deboer, gukov, tong, eberhardt, eberhardt2, eberhardt3}.    It is known that if $M^{2n}$ is a K\"ahler manifold, then the geometry of the (smooth part of the) moduli spaces $\mathscr{M}_\HE^*(M^{2n})$ is also K\"ahler manifold, see \cite{itoh} for a review. More recently, L\"ubke and Teleman \cite{lubke} have  demonstrated that for  $M^{2n}$ a Hermitian manifold, $\mathscr{M}_\HE^*(M^{2n})$  admits a strong K\"ahler with torsion (KT) structure\footnote{KT manifolds are Hermitian manifolds equipped with the unique  connection $\h\nabla$  with skew-symmetric torsion $H$ such that  $\h\nabla g=\h\nabla \omega=0$, i.e. $\h\nabla$ is {\sl compatible} with the Hermitian structure. Then, the torsion is given by $H=-\ii_I d\omega$.  The KT structure is strong, iff $H$ is closed, $dH=0$, see \cite{pw} for a recent review.}.

Hermitian-Einstein connections on four-dimensional manifolds are special because such connections satisfy  the instanton anti-self-duality condition.  The anti-self-duality condition depends on the choice of orientation on the underlying 4-dimensional manifold $M^4$ and not  on the choice of complex structure on $M^4$. This is provided  that the choice of complex structure is compatible with the  orientation used to impose the anti-self-duality condition. As a result the moduli spaces of instantons $\mathscr{M}_\asd^*(M^{4})$ can exhibit  ``exotic'' geometric structures after making judicious choices of a complex structure on $M^4$. In particular, if $M^4$ admits a hyper-K\"ahler, (oriented) bi-KT, hyper-K\"ahler with torsion (HKT) or (oriented) bi-HKT structure\footnote{Bi-KT and bi-HKT structures are also known as generalised K\"ahler and generalised hyper-K\"ahler structures \cite{hitchin, gualtieri}, respectively.}, then  $\mathscr{M}_\asd^*(M^{4})$ exhibits the same structure as well \cite{lubke, hitchin, moraru}-- in all cases the torsion 3-form\footnote{All Hermitian manifolds $(M^{2n}, g, \omega)$ admit a unique compatible connection $\h \nabla$ with skew-symmetric torsion $H$, where $H=-\ii_I d\omega$,  $\h\nabla_iY^j=D_i Y^j+\frac{1}{2} H^j{}_{ik} Y^k$ and $D$ is the Levi-Civita connection.}  $\mathcal{H}$ of the moduli space either vanishes or it is a closed 3-form, $d\mathcal{H}=0$, i.e. the KT, bi-KT, HKT and bi-HKT structures are strong.

Another aspect of the geometry of $\mathscr{M}_\HE^*(M^{2n})$ is the symmetries induced on it from the  symmetries of the underlying manifold $M^{2n}$.  These have  extensively been investigated in \cite{gp} following the results in \cite{witten} on the symmetries of $\mathscr{M}_\asd^*(S^3\times S^1)$ induced from the Lie group structure on $S^3\times S^1$.  The latter have been instrumental in providing evidence for the  AdS$_3$/CFT$_2$ mentioned above. In particular, it was demonstrated \cite{witten}, see also \cite{gp},  that the sigma models with target space $\mathscr{M}_\asd^*(S^3\times S^1)$ exhibit as a symmetry two copies of the large $N=4$ superconformal algebra of \cite{sevrin, schoutens}.

 The investigation  \cite{lubke, hitchin, moraru, witten,  gp} of the geometry and symmetries of $\mathscr{M}_\HE^*(M^{2n})$ for $(M^{2n}, g, \omega)$ a Hermitian manifold  has been made in the Gauduchon gauge, i.e. in the gauge that the Lee form\footnote{The Lee form is defined as $\theta_i=D^k\omega_{kj} I^j{}_i$, where $D$ is the Levi-Civita connection. The metric $g$ is Gauduchon, iff $D^i\theta_i=0$. For $M^{2n}$ compact, given a Hermitian metric $g$, there is another one $e^\Lambda g$, where $\Lambda$ is a function on $M^{2n}$, such that $e^\Lambda g$ satisfies the Gauduchon condition.} $\theta$ of the Hermitian geometry of $M^{2n}$ is divergence free -- if the underlying manifold is K\"ahler or hyper-K\"ahler the Lee form vanishes, $\theta=0$, and so the associated metrics are automatically in the Gauduchon gauge.

The choice of the Gauduchon gauge to investigate the geometry of $\mathscr{M}_\HE^*(M^{2n})$ is not much of a restriction on the geometry of the underlying compact manifold $M^{2n}$  as every Hermitian metric in its conformal class admits a Gauduchon metric.  Moreover, the Hermitian-Einstein condition on a connection $A$ of a vector bundle $E$, especially with gauge group $SU(r)$, depends on the conformal class $[g]$ of the metric $g$ instead of $g$ itself. The same applies for the conditions that the tangent vectors of the moduli space $\mathscr{M}^*_\HE$ obey, see section \ref{sec:pre}.
 So provided  that the geometry of $M^{2n}$ is not subject to some other constraint, one can always choose the Gauduchon metric to perform the investigation of the moduli space geometry. However, if another condition is imposed on the geometry of $M^{2n}$, this will not be the case whenever the additional condition is not preserved upon the rescaling $g\rightarrow e^\Lambda g$ of the metric.  Such a rescaling  is required to bring a Hermitian metric into the Gauduchon gauge. For example, if $(M^{2n}, g, \omega)$ admits a strong KT structure, as $H=-\ii_I d\omega$,  the closure $dH=0$ of the 3-form torsion $H$ is not preserved upon  an arbitrary rescaling of the metric and so of $\omega\rightarrow e^\Lambda \omega$.  The same applies whenever the holonomy of the connection with torsion $\h\nabla$ is restricted to lie in a proper subgroup of $U(n)$ or $(g, H)$ satisfy an additional condition that depends on $g$ and not its conformal class $[g]$.  These kind of additional restrictions on the Hermitian geometry of $M^{2n}$ are common place in the context of various applications in physics and in geometry.  Furthermore, even though the conditions that the tangent vectors of $\mathscr{M}_\HE^*(M^{2n})$ satisfy depend only on the conformal class $[g]$ of $g$, the geometry of $\mathscr{M}_\HE^*(M^{2n})$, like the metric and Hermitian form, depends on $g$ itself instead of its conformal class. As a result different choices of $g$  lead to different geometries on $\mathscr{M}_\HE^*(M^{2n})$.

The above remarks on the use of the Gauduchon gauge to investigate the geometry of  $\mathscr{M}_\HE^*(M^{2n})$ raise several questions. One question is on whether there is a  description of the geometry of $\mathscr{M}_\HE^*(M^{2n})$  for a metric $g$ on $M^{2n}$ that it is not  Gauduchon.  Another related  question is to identify the class of Hermitian geometries on $M^{2n}$ that this alternative description  of the geometry of $\mathscr{M}_\HE^*(M^{2n})$ can take place -- one expects some kind of a restriction that replaces  the Gauduchon gauge.  This restriction has to  be wide enough so that it includes the geometries that occur in various applications in physics and in geometry. If such an investigation can be carried out, the next question will be under which conditions the symmetries of the underlying manifold $M^{2n}$ can be lifted to $\mathscr{M}_\HE^*(M^{2n})$.

\underline{{\it New results}:} The objective of this article is to provide an answer to the first two questions.  Then, we shall proceed to investigate the symmetries of $\mathscr{M}_\HE^*(M^{2n})$ induced from those of the underlying manifold $M^{2n}$. The answer to the first question requires the modification of the inner product\footnote{In the investigation of geometry of $\mathscr{M}^*(M^{2n})$ so far, $\Phi$ is taken to be constant.}  on the tangent space of (irreducible) connections $\mathscr{A}^*(M^{2n})$ by introducing a ``dilaton field''  described by a scalar function $\Phi$ on $M^{2n}$ as
\be
(a_1, a_2)_{\mathscr {A}^*}=\int_{M} d^{2n}x  \,\sqrt{g}\,\,e^{-2\Phi}\,\, \langle a_1, a_2\rangle_{\mathfrak{g}}~,
\label{Aform1}
\ee
where $a_1, a_2$ are tangent vectors on $\mathscr{A}^*(M^{2n})$ at the connection $A\in \mathscr{A}^*(M^{2n})$, i.e. $a_1, a_2\in\Omega^1(P\times_{\Adj} \mathfrak{g})$,  and $g$ is the metric on the Hermitian manifold  $(M^{2n}, g, \omega)$; see \cite{gp} for an explanation of the notation and also section \ref{sec:pre}.
Then, we demonstrate that the analysis for the geometry of $\mathscr{M}_\HE^*(M^{2n})$ can be carried out provided that
\be
D^i\big( e^{-2\Phi} K^\flat_i\big)=0~,~~~K^\flat\equiv \theta+2 d\Phi~,
\label{concon}
\ee
where $K^\flat$ is the 1-form associated to the vector field $K$, i.e. $K^\flat(Y)= g(K, Y)$. This condition is the Gauduchon condition for another metric in the conformal class $[g]$ of $g$. But for a prescribed metric $g$ on $M^{2n}$, i.e. a metric that satisfies a condition on $M^{2n}$ that depends on $g$ and not its conformal class, (\ref{concon}) is a restriction on the geometry of $M^{2n}$.  This condition is satisfied by a wide range of manifolds. These include those that occur in heterotic string supersymmetric compactifications \cite{Strominger, Hullheterotic} and those that serve as target spaces   for conformally invariant sigma models with metric and $B$-field couplings, i.e. those that satisfy the string field equations \cite{friedan, CFMP, Ts, Ts2, CT}.
In the related context of Hermitian flows, for a review see \cite{GFS} and references therein, we shall demonstrate that  compact (generalised) gradient flow solitons solve (\ref{concon}).

After introducing the dilaton field,  the main result regarding the geometry of $\mathscr{M}_\HE^*(M^{2n})$ is described in the following theorem.

\begin{theorem}\label{th:skt}
Let $(M^{2n}, g, \omega)$ be a Hermitian manifold with a dilaton field $\Phi$. If $D^i(e^{-2\Phi} K^\flat_i)=0$ with $K^\flat=2d\Phi+\theta$, then
the moduli space $\mathscr{M}^*_\HE(M^{2n})$ admits a strong KT structure.
\end{theorem}

Similarly, we prove two theorems that describe the symmetries of $\mathscr{M}_\HE^*(M^{2n})$ induced from those of the underlying manifold $M^{2n}$. The first concerns the properties of the vector fields $\alpha_X$ induced on $\mathscr{M}^*_\HE(M^{2n})$ from holomorphic and Killing vector fields $X$ on $M^{2n}$.

\begin{theorem}\label{th:holkill}
Let $(M^{2n}, g, \omega)$ be a compact KT manifold with dilaton $\Phi$. Suppose that the condition (\ref{phig}), $D^i\big(e^{-2\Phi} K^\flat_i\big)=0$ with $K^\flat=\theta+2d\Phi$,  is satisfied and moreover $M^{2n}$ admits a holomorphic and Killing vector field $X$ such that $X\Phi=0$, then $\alpha_X$ is holomorphic and Killing vector field on $\mathscr{M}^*_\HE(M^{2n})$.
\end{theorem}

The second theorem concerns the properties of the vector fields  $\alpha_X$ induced by holomorphic and $\h\nabla$-covariantly constant vector fields $X$ on $M^{2n}$. Under the same assumptions on $(M^{2n}, g, \omega)$ as in the previous theorem, we have the following.

\begin{theorem}\label{th:phihnabla}
Let $X$ be $\h\nabla$-parallel and holomorphic vector field on $(M^{2n}, g, \omega)$, and $X\Phi=Y\Phi=0$ with $Y=-IX$, where $I$ is the complex structure of $M^{2n}$. Then,  $\alpha_X$ is $\h{\mathcal D}$-covariantly constant on $\mathscr{M}^*_{\HE}(M^{2n})$  provided that $X^\flat\wedge K^\flat$ is a $(1,1)$-form on $M^{2n}$, where $K^\flat=\theta+2 d\Phi$ and $\h{\mathcal D}$ is the compatible connection on $\mathscr{M}^*_{\HE}(M^{2n})$ with skew-symmetric torsion.
\end{theorem}

The proof of both these theorems is facilitated by the key Lemma \ref{le:main}.

%After introducing the dilaton field,  the main results regarding the geometry and symmetries of $\mathscr{M}^*(M^{2n})$ are described in theorems \ref{th:skt}, \ref{th:holkill} and \ref{th:phihnabla}. In the first, it is demonstrated that $\mathscr{M}^*(M^{2n})$ admits a strong KT structure. In the second, we prove that Killing and holomorphic vector fields $X$ on $M^{2n}$ induce Killing and holomorphic vector fields $\alpha_X$ on  $\mathscr{M}^*(M^{2n})$.  In theorem \ref{th:phihnabla} we demonstrate that if $X$ is Killing and holomorphic and the form $X^\flat\wedge K^\flat$ is a $(1,1)$-form on $M^{2n}$, then $\alpha_X$ is $\h{\mathcal{D}}$ covarianly constant, where $K^\flat=\theta+2 d\Phi$ and $\h{\mathcal{D}}$ is the connection with skew-symmetric torsion on $\mathscr{M}^*(M^{2n})$.

  To illustrate what kind of geometries on manifolds solve  the condition (\ref{concon}), we shall consider the Riemannian manifold $(M^d, g, H)$ with metric $g$ and a 3-form $H$ that satisfy the condition
  \be
\h R_{ij}-\frac{1}{4} P_{ij}=2\h\nabla_i X^\flat_j~,
\label{sinv}
\ee
where $\h\nabla$ is the metric connection with torsion $H$, $\h R$ is te Ricci tensor of $\h\nabla$,  $X$ is a vector field and $P_{ij}=P_{ji}$ will be specified later.  Such conditions restrict both $g$ and  $H$ and depend on $g$ instead of the conformal class $[g]$ of $g$.

There are several options to choosing $P$ and $dH$. First, for $P=0$ and $H$ closed, $d H=0$, this is the 1-loop scale invariance condition for bosonic sigma models with couplings the metric and the $B$-field, $H=dB$, or equivalently it is the condition for  steady flow (generalised) solitons \cite{OSW2, OSW, Tseytlin2, Huhu}.

Second, another choice for $P$ and $dH$ is that that occurs in heterotic geometries. For these geometries, $P$ and $dH$ are given\footnote{In perturbation theory both $P$ and $dH$ are corrected order by order in $\alpha'$.  Here, only the first order  correction to the expansion in $\alpha'$  is given.} as
\be
P_{ij}=-\frac{\alpha'}{4} \big(\tilde R_{ikmn} \tilde R_j{}^{kmn}- F(C)_{ik \ma\mb} F(C)_{j}{}^{k\ma\mb}\big)~,
\label{pha}
\ee
and
\be
dH=\frac{\alpha'}{4}\big(\tilde R^i{}_j\wedge \tilde R^j{}_i- F(C)^\ma{}_\mb\wedge F(C)^\mb{}_\ma\big)~,
\label{dha}
\ee
respectively, where $\alpha'$ is a parameter,  $\tilde R$ is a curvature of the tangent bundle of $M^d$ and $F$ is the curvature of a vector bundle $E$ over $M^d$.  Then (\ref{sinv}) is the condition for scale invariance of heterotic sigma model\footnote{For the scale invariance of heterotic sigma models, the condition (\ref{sinv}) has to be supplemented with an additional condition associated to the connection $C$ coupling of the gauge sector.} \cite{GHMR, HullWitten} up to and including 2-loops in perturbation theory.
There are other choices for $P$ and $H$ in (\ref{sinv}), apart from those two mentioned above,  and there are several geometries that satisfy this condition that we shall explore later\footnote{Explicit examples include Hermitian homogeneous spaces whose canonical connection has holonomy included in $SU(n)$ and its torsion is a 3-form \cite{OP}. For these $P$ and  $dH$ do not vanish and satisfy (\ref{sinv}) but they do not satisfy the condition for the heterotic sigma model scale invariance as they do not exhibit a free parameter $\alpha'$ and  $P$ and $dH$ are not corrected to all orders in perturbation theory as it is expected. Furthermore, it is not expected that $dH$ can be written as in (\ref{dha}) -- such a rewriting requires a case by case investigation.  }.

A related condition to (\ref{sinv}) that the geometries that we shall consider satisfy is
\be
\h R_{ij}-\frac{1}{4} P_{ij}+2\h\nabla_i\partial_j\Phi=0~,
\label{cinv}
\ee
 where the scalar function $\Phi$ is the dilaton.  As in the previous case, there are several cases to consider. First,  for $P=0$ and $H$ closed, $d H=0$, this is the condition required on the couplings $g$ and  $B$-field, $H=dB$, for a bosonic sigma model to be conformally invariant at 1-loop or equivalently it is the condition satisfied by gradient flow solitons, see e.g. \cite{GFS}.  Again, if $P$  is given as in (\ref{pha}) with $dH$ is expressed as in (\ref{dha}), then (\ref{cinv}) is the condition for conformal invariance of heterotic sigma model\footnote{For the conformal invariance of heterotic sigma models, the condition (\ref{cinv}) has to be supplemented with an addition condition associated to the connection coupling of the gauge sector.} up to and including 2-loops in perturbation theory.

The relationship between  (\ref{sinv}) and  (\ref{cinv}) has been explored from different perspectives. From the perspective of scale vs conformal invariance, Polchinski \cite{Polchinski} demonstrated that under certain conditions,  which include unitarity and a discrete spectrum of operator scaling  dimensions, that two-dimensional scale invariant Quantum Field Theories are actually conformally invariant. In the context of two-dimensional sigma models $P=dH=0$, the relations between (\ref{sinv}) and (\ref{cinv}) has been investigated by Hull and Townsend in \cite{townsend}. In the same context, it has been demonstrated in \cite{pw} via a Perelman style of argument \cite{Perelman} that if the sigma model target space is compact, then   (\ref{sinv}) implies  (\ref{cinv}). Equivalently, steady flow (generalised) solitons are in fact gradient flow (generalised) solitons \cite{GFS}. A similar  result can be established for heterotic sigma models with compact target space \cite{gp2}, i.e. for sigma models with $P$ and $dH$  given in (\ref{pha}) and (\ref{dha}), respectively.

Suppose that (\ref{sinv}) implies (\ref{cinv}) either as a consequence  of one of the results mentioned above or as an additional assumption on the geometry of $(M,g,H)$.  Then  consistency between (\ref{sinv}) and (\ref{cinv}) requires that
\be
\h\nabla_i V^j=0~,~~~ V^\flat_i\equiv X^\flat_j+2\partial_j \Phi~,
\label{concon2}
\ee
i.e. the vector field $V$ is $\h\nabla$-covariantly constant. In turn this implies that $V$ is Killing vector field and $\ii_V H=dV^\flat$.

To see the relevance of the geometries mentioned above in the context  of the moduli space of  Hermitian-Einstein connections $\mathscr{M}^*_\HE$, it is remarkable to observe that $V$ appears in the key condition  (\ref{concon}) required for the investigation of the geometry of moduli spaces.  This is provided that one identifies $X^\flat=\theta$, i.e. $V=K$.  Indeed, we shall demonstrate in section \ref{sec:ex} that for compact Calabi-Yau manifolds with torsion (CYT), i.e. KT manifolds that the holonomy of $\h\nabla$ is included in $SU(n)$, $X^\flat$ can be identified with $\theta$, $X^\flat=\theta$,  and the condition (\ref{concon2}) implies (\ref{concon}).

Because of the key role that the vector field $K$ has both in exploring the properties of $\mathscr{M}_\HE^*(M^{2n})$ and those of $M^{2n}$, we provide a systematic study of its properties for various geometries on $M^{2n}$ and we demonstrate that it can be both Killing and holomorphic. The main result in the context of KT geometry is described in Proposition \ref{prop:hol} -- this is an extension of the proposition 4.1 of Streets and Ustinovskiy in \cite{jsyu}.  The main applications of this result on the geometry of moduli spaces $\mathscr{M}_\HE^*(M^{2n})$ are given in Theorem \ref{th:modkill}.

\underline{{\it Organisation}:} The paper is organised as follows.  In section 2, after an introduction that summarises  the geometry of $\mathscr{M}_\HE^*(M^{2n})$ provided that the metric on $M^{2n}$ is in the Gauduchon gauge, we introduce the modifications in the metric and Hermitian form on
$\mathscr{M}_\HE^*(M^{2n})$ in the presence of a dilaton. Then, we proceed with the proof of the Theorem \ref{th:skt} and demonstrate  that $\mathscr{M}_\HE^*(M^{2n})$ admits a strong KT structure with respect to the new metric and new Hermitian form that now depend on the dilaton.  We also explain the necessity of the condition (\ref{concon}) for the proof of the result.

In section 3, we demonstrate that holomorphic and Killing vector fields on $M^{2n}$ induce holomorphic and Killing vector fields on $\mathscr{M}_\HE^*(M^{2n})$.
The proof of this statement in this section is described in the Gauduchon gauge.  The reason behind this is to simplify the intermediate  steps of the proof. It also makes it easier to compare these results with those in \cite{gp} that have been demonstrated under stronger assumptions.

In section 4, the results of section 3 are adapted to geometries on $M^{2n}$ that do not satisfy the Gauduchon gauge and  include a dilaton.  This section concludes with the proof of the Theorems \ref{th:holkill} and \ref{th:phihnabla} stated above.

In section 5, we explore the symmetries of KT manifolds $(M^{2n}, g,\omega)$ that satisfy (\ref{cinv}). Then, we use these  to explore the kind of geometries that satisfy the condition (\ref{concon}).  We also describe the consequences of these results on the symmetries of the moduli space of Hermitian-Einstein connections $\mathscr{M}_\HE^*(M^{2n})$.

\section{Geometry of moduli space of Hermitian-Einstein connections}

\subsection{Preliminaries}\label{sec:pre}

Before we proceed with the proof of our main results, we set up our notation and give some key definitions that will be used later.  For notation, we follow closely that of \cite{gp} where some more explanation can be found. We also summarise, without a proof, some key steps in the description of the geometry of the moduli space of Hermitian-Einstein connections. This section concludes with  the main result of L\"ubke and Teleman \cite{lubke} presented in Theorem \ref{th:lt} on the geometry of the moduli space of Hermitian-Einstein connections in the Gauduchon gauge.

\subsubsection{The space of connections}
Let $P(M^d, G)$ be a principal bundle with fibre group $G$ and base space $M^d$ and denote the space of all connections on $P(M^d, G)$ with $\mathscr{A}$.
These connections have a gauge group $G$ and the space of all gauge transformations, $\mathscr{G}$,  is the space of sections of the associated bundle $P\times_{\Adj} G$ with respect to the adjoint representation, i.e. $\mathscr{G}=\Omega^0(P\times_{\Adj} G)$. Locally the gauge transformations are maps from $M^d$ to $G$.  The Lie algebra of $\mathscr{G}$ is $\Omega^0(P\times_{\Adj} \mathfrak{g})$, where $\mathfrak{g}$ is the Lie algebra of $G$.

The space of connections $\mathscr{A}$ is an affine space and $\Omega^1(P\times_{\Adj} \mathfrak{g})$ acts on it with translations as
\begin{align}
\Omega^1(P\times_{\Adj} \mathfrak{g})\times \mathscr{A} &\rightarrow \mathscr{A}
\cr
(a, A)\qquad&\rightarrow A+a~,
\end{align}
where locally $\Omega^1(P\times_{\Adj} \mathfrak{g})$ is the space of Lie algebra valued $1$-forms on $M^d$.
 The tangent space of $\mathscr{A}$ at every point $A\in \mathscr{A}$ can be identified with $\Omega^1(P\times_{\Adj} \mathfrak{g})$. Moreover, the tangent directions along the orbit of $\mathscr{G}$ that passes through $A$ at $A$ in $\mathscr{A}$ are given by
\be
a=d_A \epsilon\equiv  d\epsilon+ [A,\epsilon]_{\mathfrak{g}}~,
\ee
for every $\epsilon \in \Omega^0(P\times_{\Adj} \mathfrak{g})$ with $[\cdot, \cdot]_{\mathfrak{g}}$ the commutator\footnote{In what follows, we shall encounter several brackets. As it has already been mentioned $[\cdot, \cdot]_{\mathfrak{g}}$ is the commutator of Lie algebra $\mathfrak{g}$.   We reserve $[\cdot, \cdot]$ as the commutator of two vector fields on $M^d$ and $\lsq\cdot, \cdot\rsq$ as the commutator of two vector fields on the space of connections and associated moduli spaces. The latter brackets will be defined below.} of the Lie algebra  $\mathfrak{g}$.

If $M^{2n}$, $d=2n$,  is a Hermitian manifold with metric $g$ and complex structure $I$, then $\mathscr{A}$ admits a complex structure $\mathcal{I}$ given by
\be
\mathcal{I}(a)=-\ii_I a~,
\ee
where $(\ii_I a)_i= I^j{}_i a_j$.
Moreover, $\mathscr{A}$ is a K\"ahler manifold with metric given by
\be
(a_1, a_2)_{\mathscr{A}}=\int_{M^{2n}} d^dx~ \sqrt{g} g^{ij} \langle a_{1i}, a_{2j}\rangle_{\mathfrak{g}}~.
\label{Aherm}
\ee
The K\"ahler form is
\be
\Omega_{\mathscr{A}}(a_1, a_2)\equiv ( a_1, \mathcal{I} a_2)_{\mathscr{A}} =\frac{1}{(n-1)!}\, \int_{M^{2n}} \omega^{n-1}\wedge  \langle a_1\wedge  a_2\rangle_{\mathfrak{g}}~.
\label{Aform}
\ee
The closure of $\Omega_{\mathcal{A}}$ can easily be seen as it is a constant form, i.e. it does not depend on the point $A\in \mathscr{A}$ that it is defined on or equivalently is invariant under the action of $\Omega^1(P\times_{\Adj} \mathfrak{g})$.

\subsubsection{The tangent space of Hermitian-Einstein connections}

Let $(M^{2n}, g, \omega)$ be a Hermitian manifold. The Hermitian-Einstein condition on a connection $A$ with gauge group $G$ of a vector bundle $E$  expressed in the adjoint representation  is
\be
F^{2,0}=F^{0,2}=0~,~~~~\omega\llcorner F\equiv\frac{1}{2} \omega^{ij} F_{ij}= \lambda~,
\label{HEcond1}
\ee
where $\lambda$ is a real constant, which commutes with all the other elements of the gauge group Lie algebra $\mathfrak{g}$.  $G$ should be thought as either $U(r)$, $SU(r)$ or $PU(r)=SU(r)/\bZ_r$ but  we shall carry out the calculations below  without specifying $G$ -- If the gauge group is either $SU(r)$ or $PU(r)$, then $\lambda=0$.  For $\lambda=0$, (\ref{HEcond1}) depends on the conformal class $[g]$ of the Hermitian metric $g$ instead of $g$ itself.

Let $\mathscr{A}^*_{\HE}$ be  the  space of Hermitian-Einstein connections of a complex vector bundle $E$ \footnote{We shall also use $\mathscr{A}(M^{2n})$ to denote the space of connections $\mathscr{A}$ whenever it is required to distinguish spaces for either clarity or economy in the description. A similar notations will be adapted for the moduli spaces $\mathscr{M}$ described below.} associated to the principal bundle $P(M^{2n}, G)$ on which the gauge group $\mathscr{G}^*$ acts freely.
For $G=SU(r)$, $\mathscr{G}^*= \Omega^0(P\times_\Adj SU(r))$ and for $G=U(r)$, $\mathscr{G}^*= \Omega^0(P\times_\Adj U(r))/U(1)$ as in the latter case the centre $U(1)$ of $\Omega^0(P\times_\Adj U(r))$ is in the isotropy group of every connection on $P$. Clearly,  $\mathscr{A}^*_{\HE}$ is a subspace of $\mathscr{A}$.

 It follows from (\ref{HEcond1}) that the tangent space of $\mathscr{A}^*_{\HE}$ at every point $A\in \mathscr{A}^*_{\HE}$ is spanned by the elements  $a\in\Omega^1(P\times_{\Adj} \mathfrak{g})$ that satisfy
\be
d_A a^{2,0}=d_A a^{0,2}=0~,~~~\omega \llcorner d_A a=0~,
\label{tangentcon}
\ee
i.e. those elements of $\Omega^1(P\times_{\Adj} \mathfrak{g})$ that their (covariant) exterior derivative $d_A a$ is a $(1,1)$-form on $M^{2n}$ and $\omega$-traceless.  Clearly, the conditions (\ref{tangentcon}) on the tangent vectors depend on the conformal class $[g]$ of $g$.
Again the vectors tangent to the  orbit of $\mathscr{G}$ through a point $A\in \mathscr{A}$ are
\be
a=d_A \epsilon~,
\ee
where $\epsilon\in \Omega^0(P\times_{\Adj} \mathfrak{g})$. As the gauge group acts freely on $\mathscr{A}^*_{\HE}$, the only solution to the equation $d_A \epsilon=0$ is $\epsilon=0$.

It turns out that $\mathscr{A}^*_{\HE}$ is not a Hermitian manifold with the induced Hermitian structure from $\mathscr{A}$ described above in the previous section.
The reason is that if the tangent vector $a$ satisfies the second condition in (\ref{tangentcon}), $\mathcal{I}(a)$ does not necessarily satisfy the same condition. However, this will be resolved below in the description of the tangent space of the moduli space of Hermitian-Einstein connections.

\subsubsection{The tangent space of the moduli space of Hermitian-Einstein connections}

The smooth part of the moduli space of Hermitian-Einstein connections is $\mathscr{M}^*_{\HE}= \mathscr{A}^*_{\HE}/\mathscr{G}^*$.  The element $[A]$ of $\mathscr{M}^*_{\HE}$ is the orbit of $\mathscr{G}^*$ through the connection $A$.  The tangent vectors $\alpha$ of $\mathscr{M}^*_{\HE}$ at $[A]$ are described  via their horizontal lifts to vectors $a^h$ on  $\mathscr{A}^*_{\HE}$, i.e. $a^h$ satisfy (\ref{tangentcon}), and in addition the horizontality condition, or equivalently the gauge fixing condition,
\be
\omega \llcorner d_A \ii_Ia=0 \Longrightarrow D_A^i a^h_i+\theta^i a^h_i=0~,
\label{hor}
\ee
where $\theta$, $\theta_i=D^j \omega_{jk} I^k{}_i$,  is the Lee form of $M^{2n}$ and $D$ the Levi-Civita connection of $g$.  This horizontality condition has the property that if the tangent vector $a$ satisfies the second condition in (\ref{tangentcon}), then $\mathcal{I}(a)=-\ii_Ia$ also satisfies the same condition. Notice that again (\ref{hor}) depends on the conformal class $[g]$ of $g$. As this is the case for the conditions on the tangent vectors of $\mathscr{A}^*_\HE$ (\ref{tangentcon}), one concludes that the conditions on the tangent vectors of $\mathscr{M}^*_\HE$ depend on the conformal class $[g]$ of the Hermitian metric $g$ instead of $g$ itself.

The decomposition of a tangent vector $a$ as
\be
a=a^h+ a^v=a^h+ d_A \epsilon~,
\label{split}
\ee
is unique with $\epsilon$ satisfying the differential equation
\be
\mathcal{O} \epsilon=-\omega\llcorner d_A^c d_A\epsilon =(D_A^iD_{Ai}+\theta^i  D_{Ai})\epsilon=-\omega\llcorner d_A^c a~.
\label{oop}
 \ee
This is because  $\mathcal{O}$ has trivial kernel and  is onto provided that
\be
D^i \theta_i=0~,
\ee
i.e. $g$ is a Gauduchon metric. Therefore, $\mathcal{O}$ is invertible.  The splitting of tangent vectors (\ref{split}) defines a connection on $\mathscr{A}^*_{\HE}$, where the horizontal component of $a$ is $a^h$ and the vertical component of $a$ is $a^v=d_A\epsilon$ -- this justifies the terminology that has been used for $a^h$.

The main result of \cite{lubke} on the geometry of $\mathscr{M}^*_{\HE}$ can be described as follows:
\begin{theorem}\label{th:lt}
 $\mathscr{M}^*_{\HE}$ admits a strong KT structure with  metric and Hermitian form  defined as
\be
(\alpha_1, \alpha_2)_{\mathscr{M}^*_{\HE}}\equiv (a^h_1, a^h_2)_{\mathscr{A}^*_{\HE}}~,~~~\Omega_{\mathscr{M}^*_{\HE}}(\alpha_1, \alpha_2)\equiv \Omega_{\mathscr{A}^*_{\HE}}(a^h_1, a^h_2)~,
\label{innerform}
\ee
where $a_1^h$ and $a_2^h$ are the horizontal lifts of the tangent vectors  $\alpha_1$ and $\alpha_2$ of $\mathscr{M}^*_{\HE}$ at $[A]$, respectively, and $(\cdot, \cdot)_{\mathscr{A}^*_{\HE}}$ and $\Omega_{\mathscr{M}^*_{\HE}}(\cdot, \cdot)$ is the restriction of the metric (\ref{Aherm})  and Hermitian form (\ref{Aform}) of $\mathscr{A}$  on $\mathscr{A}^*_{\HE}$. This means that $\mathscr{M}^*_{\HE}$ is a Hermitian manifold admitting a compatible $\h{\mathcal D}$ connection with  torsion ${\mathcal H}$ that is a closed 3-form.
\end{theorem}

\begin{remark}
Even though the tangent vectors of $\mathscr{M}^*_{\HE}(M^{2n})$ depend only on the conformal class $[g]$ of the metric $g$ of the underlying Hermitian manifold $M^{2n}$, the geometry of $\mathscr{M}^*_{\HE}(M^{2n})$, like the metric and Hermitian form in (\ref{innerform}), depend on $g$ itself instead of its conformal class. Thus the geometry of $\mathscr{M}^*_{\HE}(M^{2n})$ is sensitive to the choice of the metric on $M^{2n}$ and so different choices for $g$ lead to different geometries on $\mathscr{M}^*_{\HE}(M^{2n})$.
\end{remark}

\subsection{Geometry of moduli spaces and the dilaton field}

\subsubsection{Metric and Hermitian form on the moduli space}

So far, the geometry of the Hermitian-Einstein connection moduli spaces has been investigated in the Gauduchon gauge, i.e. in the gauge that the Lee form $\theta$ of the metric on $M^{2n}$ is divergence free $D^i\theta_i=0$. This is not a restriction on the geometry of $(M^{2n}, g, \omega)$ as the conformal class of every Hermitian metric $g$ admits such a metric. However, in many applications, especially in the context of string theory and in manifolds with special holonomy, the relevant metric on $(M^{2n}, g, \omega)$ is not in the Gauduchon gauge and so $D^i\theta_i\not=0$.  We shall overcome this difficulty by introducing a scalar field $\Phi$ on $M^{2n}$, the dilaton. As a result, we shall investigate the geometry of the moduli space of Hermitian-Einstein connections in a way that incorporates the dilaton $\Phi$. Luckily, this can be done as an adaptation of the results that have already derived in the Gauduchon gauge.  As we shall point out, there is a class of geometries that such a scalar field naturally arises and these are precisely those that appear in the context of string theory, manifolds with special holonomy and  geometric flow gradient (generalised) solitons.

To begin consider a Hermitian manifold $(M^{2n}, g, \omega)$ with metric $g$ and form $\omega$, where $g$ may not be  in the Gauduchon gauge. Suppose that $M^{2n}$ also admits a scalar function $\Phi$. The Hermitian-Einstein condition (\ref{HEcond1}) is taken with respect to the complex structure $I$ and Hermitian form $\omega$ of $M^{2n}$. As a result, the tangent vectors $a$ of $\mathscr{A}^*_\HE$ satisfy the conditions (\ref{tangentcon})
but with a Hermitian form $\omega$ whose associated metric $g$  is not Gauduchon.  Moreover for $\mathscr{M}_\HE$ to admit a complex structure, the gauge fixing, or horizontality condition,  should be chosen as in (\ref{hor}), again with respect to the metric  $g$ (and Hermitian form $\omega$) that do not satisfy the Gauduchon condition.
Clearly, the conditions on the tangent vectors of $\mathscr{A}^*_\HE$ to be along $\mathscr{A}^*_\HE$ and horizontal are the same as those we have investigated so far. This potentially can allow us to decompose the tangent vectors $a$ of $\mathscr{A}^*_\HE$ into horizontal $a^h$ and vertical components $a^v$ as
\be
a=a^h+ a^v~,
\label{decom}
\ee
where $a^h$ satisfies (\ref{hor}) and $a^v=d_A\epsilon$ as the vectors tangent to the orbits of the gauge group in $\mathscr{A}^*_\HE$ are spanned by gauge transformations. However for the decomposition in (\ref{decom}) to define a connection on $\mathscr{A}^*_\HE$ it has to be unique -- for the proof of this statement,  the Gauduchon condition on the metric $g$ is traditionally needed.  In the absence of this, another approach is required.

One way to incorporate the dilaton in the analysis of the geometry on the moduli spaces on non-Gauduchon Hermitian manifolds $(M^{2n}, g, \omega)$ is to modify the inner product, and consequently the Hermitian form, on
$\mathscr{A}$ as
\begin{align}
(a_1, a_2)_{\mathscr{A}}&\equiv \int_{M^{2n}} d^n x \sqrt g\,\,e^{-2\Phi}\,\, g^{-1}\, \langle a_1, a_2\rangle_{\mathfrak{g}}=\int_{M^{2n}} d^n x \sqrt g\,\,e^{-2\Phi}\,\, g^{ij}\, \langle a_{1i}, a_{2j}\rangle_{\mathfrak{g}}~,
\cr
\Omega_{\mathscr{A}}(a_1, a_2)&\equiv ( a_1, \mathcal{I} a_2)_{\mathscr{A}} =\frac{1}{(n-1)!}\, \int_{M^{2n}} \,\,e^{-2\Phi}\,\, \omega^{n-1}\wedge  \langle a_1\wedge  a_2\rangle_{\mathfrak{g}}~,
\label{Aform2}
\end{align}
see \cite{gp} for the notation in the last equation.
Assuming that the decomposition (\ref{decom}) is unique, the inner product and Hermitian form on the moduli space $\mathscr{M}^*_\HE$ can be defined as
\begin{align}
(\alpha_1, \alpha_2)_{\mathscr{M}^*_\HE}&\equiv (a^h_1, a^h_2)_{\mathscr{A}^*_\HE}~,
\cr
\Omega_{\mathscr{M}^*_\HE}(\alpha_1, \alpha_2)&\equiv \Omega_{\mathscr{A}^*_\HE}(a^h_1, a^h_2)~,
\end{align}
where $a_1^h$ and $a_2^h$ are the horizontal lifts of the vectors $\alpha_1$ and $\alpha_2$ on $\mathscr{M}^*_\HE$, respectively,  $(\cdot, \cdot)_{\mathscr{A}^*_\HE}$ and $\Omega_{\mathscr{M}^*_\HE}(\cdot, \cdot)$ is the restriction of those on $\mathscr{A}$ in (\ref{Aform2}) onto $\mathscr{A}^*_\HE$.

\begin{theorem}\label{th:two}
Let $(M^{2n}, g, \omega)$ be a (non-Gauduchon) Hermitian manifold with dilaton field $\Phi$ and Lee form $\theta$. The decomposition of tangent vector $a$ on $\mathscr{A}^*_\HE$ in horizontal  and vertical components as $a=a^h+ a^v$  is unique provided that
\be
D^i\big(e^{-2\Phi} K^\flat_i\big)= \big(D^i (e^{-2\Phi}\theta_i)-D^2 e^{-2\Phi}\big)=0~,
\label{phig}
\ee
where $a^h$ satisfies the horizontality condition (\ref{hor}), $K^\flat=2 d \Phi+ \theta$ and $a^v=d_A\epsilon$.
\end{theorem}
\begin{proof}
Applying the horizontality condition (\ref{hor}) on $a$, we find that
\be
D_A^i a_i+ \theta^i a_i=(D^i_AD_{Ai}+\theta^i D_{Ai}) \epsilon~,
\ee
as $a^v=d_A\epsilon$. To prove that the decomposition $a=a^h+ a^v$  is unique, one has to show that the operator
\be
\mathcal{O}=D^i_AD_{Ai}+\theta^i D_{Ai}~,
\ee
is invertible.

First, the kernel of $\mathcal{O}$ is trivial. Indeed, suppose that there is $\epsilon\not=0$ in the kernel of $\mathcal{O}$. Then,
\begin{align}
&\int_M  d^{2n}x\,\sqrt{g}\,\, e^{-2\Phi}\,\, \langle \epsilon, (D_A^i(d_A)_i+\theta^i  (d_A)_i) \epsilon\rangle_{\mathfrak{g}}
\cr
&
\qquad =-\int_M d^{2n}x\,\sqrt{g}\,\, e^{-2\Phi}\,\, \Big(\langle d_A^i\epsilon, (d_A)_i\epsilon \rangle_{\mathfrak{g}}- \partial^i\Phi \partial_i\langle\epsilon, \epsilon \rangle_{\mathfrak{g}}-{1\over2} \theta^i \partial_i \langle\epsilon,  \epsilon\rangle_{\mathfrak{g}}\Big)
\cr
&\qquad =-\int_M d^{2n}x\,\sqrt{g}\, \Big( e^{-2\Phi} \langle d_A^i\epsilon, (d_A)_i\epsilon \rangle_{\mathfrak{g}}+ D_i\big(e^{-2\Phi}(\partial^i\Phi+{1\over2} \theta^i)\big) \langle\epsilon, \epsilon \rangle_{\mathfrak{g}}\Big)=0~.
\end{align}
Using the hypothesis of the theorem, the above identity implies that $d_A\epsilon=0$. As for irreducible connections $A$ $d_A\epsilon=0$ implies that $\epsilon=0$, the kernel of $\mathcal{O}$ is trivial.

It remains to show that $\mathcal{O}$ is onto. Suppose that it is not and there is $\eta$ that is orthogonal to the image of $\mathcal{O}$. Using the hypothesis of the theorem (\ref{phig}), this implies that
\begin{align}
\int_M \sqrt{g} d^{2n}x\, e^{-2\Phi} &\langle \eta, (D_A^iD_{A i}+\theta_i  D^i_{A}) \epsilon\rangle_{\mathfrak{g}}
\cr
&
=\int_M \sqrt{g} d^{2n}x\, e^{-2\Phi} \langle \big(D_A^iD_{Ai}-(\theta_i+4 \partial_i\Phi)  D_A^i\big)\eta, \epsilon\rangle_{\mathfrak{g}}=0~,
\end{align}
where the operator $\mathcal{O}^\dagger= D_A^iD_{Ai}-(\theta_i+4 \partial_i\Phi)  D_A^i$ is the adjoint of $\mathcal{O}$.  Thus $\eta$ is in the kernel of $\mathcal{O}^\dagger$. However, $\mathcal{O}^\dagger$ has a trivial kernel.  This computation is similar to that presented for $\mathcal{O}$. Therefore we conclude that $\mathcal{O}$ is invertible and the decomposition of the tangent space of $\mathscr{A}^*_{\HE}$ into horizontal and vertical subspaces is unique.
\end{proof}

\begin{remark}
The condition (\ref{phig}) is satisfied provided that the 1-form $K^\flat=\theta+2d\Phi$ is Killing and $K\Phi=0$.  As we shall demonstrate this condition allows for the inclusion in this analysis  of the geometries  that occur in string theory, those on manifolds with special holonomy and those of (generalised) solitons of Hermitian geometric flows. Moreover, (\ref{phig}) can possibly be weaken further, e.g if $D^i\big(e^{-2\Phi}(\partial_i\Phi+{1\over2} \theta_i)\big)\geq 0$, the kernel of $\mathcal{O}$ is trivial. However  in what follows, the stronger version of the condition stated in the theorem above will suffice.
\end{remark}

\subsubsection{Geometry of moduli spaces}

The investigation of the geometry of $\mathscr{M}^*_\HE$ in the presence of the dilaton field $\Phi$ can be carried out with a few  changes from that that has been described for Hermitian manifolds equipped with the Gauduchon metric \cite{lubke}.  Because of this, we shall not elaborate on the proof. Instead, we shall give the statement and only emphasise the differences in the proof when they arise.

Before we proceed with the proof of one of the main theorems, let $a_2^h$ be a horizontal vector field $D_A^i a_{2i}+\theta^i a_{2i}=0$.  Taking the directional derivative\footnote{The directional derivative of a function $f$ on $\mathscr{A}$ along the vector field $a$  is defined as $a\cdot f(A)\equiv \frac{d}{dt} f(A+ t a)\vert_{t=0}$. The commutator of two vector fields $a_1$ and $a_2$ on $\mathscr{A}$ is then defined as $\lsq a_1, a_2\rsq f\equiv a_1 \cdot (a_2\cdot f)- a_2\cdot(a_1\cdot f)$.} of this condition along $a^h_1$ and then the commutator of $a^h_1$ and $a^h_2$, one finds that
\be
2g^{ij} [a_{1i}, a_{2j}]_{\mathfrak{g}}+D_A^i\lsq a_1, a_2\rsq_i+\theta^i \lsq a_1, a_2\rsq_i=0\Longrightarrow D_A^i\lsq a_1, a_2\rsq^v_i+\theta^i \lsq a_1, a_2\rsq^v_i=-2g^{ij} [a_{1i}, a_{2j}]_{\mathfrak{g}}
\ee
where  $\lsq a_1, a_2\rsq^v$ is the vertical component of the commutator. As this is a vector field along the orbits of the gauge group, there is a $\Theta$ such that
$\lsq a^h_1, a^h_2\rsq^v=-d_A\Theta(a_1^h, a_2^h)$ with
\be
\mathcal{O}\Theta(a_1^h, a_2^h)=2g^{ij} [a^h_{1i}, a^h_{2j}]_{\mathfrak{g}}~.
\label{diftheta}
\ee
$\Theta$ is interpreted as the curvature associated with the horizontality condition (\ref{hor}). As expected it is a Lie algebra valued 2-form on $\mathscr{M}^*_{\HE}$, where the Lie algebra is that of $\mathscr{G}^*$. The proof of the theorem below is structured along the lines of that presented in \cite{gp} for Gauduchon manifolds.

%\begin{theorem}\label{th:skt} Let $(M^{2n}, g, \omega)$ be a Hermitian manifold with a dilaton field $\Phi$. If $D^i(e^{-2\Phi} K^\flat_i)=0$ with %$K^\flat=2d\Phi+\theta$, then
%the moduli space $\mathscr{M}^*_\HE$ admits a strong KT structure.
%\end{theorem}
\vskip 0.3cm
%\begin{theorem}*
{\bf{ Theorem}} \ref{th:skt} stated at the introduction.
%\end{theorem}*
\begin{proof}
To prove this, let us first compute the exterior derivative of the Hermitian form $\Omega_{\mathscr{M}^*_\HE}$.  A computation similar to that performed in the Gauduchon gauge reveals that
\begin{align}
d\Omega_{\mathscr{M}^*_{\HE}}&(\alpha_1, \alpha_2, \alpha_3)=-\frac{1}{(n-1)!} \,\int_{M^{2n}}\,\, e^{-2\Phi}\,\, \omega^{n-1}\wedge \langle d_A\Theta (a_1^h, a_2^h) \wedge  a^h_3\rangle_{\mathfrak{g}}+ \mathrm{cyclic~ in~} (a^h_1, a^h_2, a^h_3)
\cr
&=\frac{1}{(n-1)!} \,\int_{M^{2n}} d(e^{-2\Phi}\,\,\omega^{n-1})\wedge \langle \Theta (a_1^h, a_2^h),   a^h_3\rangle_{\mathfrak{g}} + \mathrm{cyclic~ in~} (a^h_1, a^h_2, a^h_3)
\label{domega}
\end{align}
Using this, one can show that the 3-form torsion is given by
\begin{align}
\mathcal {H}_{\mathscr{M}^*_{\HE}}&(\alpha_1, \alpha_2, \alpha_3)=-d^c\Omega_{\mathcal{M}^*_{\HE}}(\alpha_1, \alpha_2, \alpha_3)=-\ii_{\mathcal{I}}  d\Omega_{\mathcal{M}^*_{\HE}}(\alpha_1, \alpha_2, \alpha_3)
\cr
&
=-\frac{1}{(n-1)!} \,\int_{M^{2n}} d^c\big( e^{-2\Phi}\,\,\omega^{n-1}\big)\wedge \Big(\langle \Theta (a_1^h, a_2^h),   a^h_3\rangle_{\mathfrak{g}} + \mathrm{cyclic~ in~} (a^h_1, a^h_2, a^h_3)\Big)
\cr
&
=- \,\int_{M^{2n}}d^{2n}x\,\sqrt{g}\,\,e^{-2\Phi}\,\, (\theta_i+2\partial_i\Phi)\, \, g^{ij}\,  \Big(\langle \Theta (a_1^h, a_2^h),   a^h_{3j}\rangle_{\mathfrak{g}} + \mathrm{cyclic~ in~} (a^h_1, a^h_2, a^h_3)\Big)~.~.
\label{mtorsion}
\end{align}
As $\Theta$ is a $(1,1)$-form,  $\mathcal {H}_{\mathscr{M}^*_{\HE}}$ is a $(2,1)\oplus (1,2)$-form.  Moreover, one can demonstrate that by construction $d\Omega_{\mathscr{M}^*_{\HE}}=\ii_{\mathcal{I}}\mathcal {H}_{\mathscr{M}^*_{\HE}}$.  These prove that
\be
\h{\mathcal{D}} \mathcal{I}=0
\ee
and so $\mathscr{M}^*_{\HE}$ admits a KT structure.

It remains to show that $\mathcal {H}_{\mathscr{M}^*_{\HE}}$ is closed.  For this, one has to compute the exterior derivative of $\mathcal {H}_{\mathscr{M}^*_{\HE}}$. A straightforward computation reveals, using the Bianchi identity for $\Theta$, that
\begin{align}
d \mathcal {H}_{\mathcal{M}^*_{\HE}}(\alpha_1, \alpha_2, \alpha_3, \alpha_4)
&=\int_{M^{2n}}d^{2n}x\,\sqrt{g}\,\,e^{-2\Phi}\,\, (\theta_i+2\partial_i\Phi)\, \, g^{ij}\,   \Big(D_j \big(\langle \Theta (a_1^h, a_2^h),  \Theta(a_3^h, a_4^h)\rangle_{\mathfrak{g}}\big)
\cr
&\qquad\qquad\qquad + \mathrm{cyclic~ in~} (a^h_1, a^h_2, a^h_3)\Big)
\cr
&=-\int_{M^{2n}}d^{2n}x\,\sqrt{g}\,\,D^i\Big(e^{-2\Phi}\,\, (\theta_i+2\partial_i\Phi)\Big)  \Big(\langle \Theta (a_1^h, a_2^h),  \Theta(a_3^h, a_4^h)\rangle_{\mathfrak{g}}
\cr
&
\qquad\qquad \qquad + \mathrm{cyclic~ in~} (a^h_1, a^h_2, a^h_3)\Big)=0~,
\label{dhm}
\end{align}
where we have used the hypothesis  of the theorem.
\end{proof}

\begin{remark}
Conformally balanced is a class of Hermitian manifolds $(M^{2n}, g, \omega)$ for which the Lee form is an exact 1-form, i.e. $\theta=-2d\Phi$. Clearly for those, $\mathscr{M}^*_\HE$ is a K\"ahler manifold -- observe that the torsion ${\mathcal H}$ in (\ref{mtorsion})  vanishes . Such Hermitian manifolds occur in supersymmetric string compactifications -- the conformal balanced condition is a requirement imposed on the geometry of the internal space $M^{2n}$ as a consequence of the dilatino Killing spinor equation.  Moreover, the gaugino Killing spinor equation implies that the connection of the gauge sector satisfies the Hermitian-Einstein condition. Therefore, one consequence of our results is that for fixed metric and complex structure on the internal space, the part of the moduli space of such compactifications due to the choice of a Hermitian-Einstein connection is a K\"ahler manifold, see also \cite{mcorist1, mcorist2, mcorist3}.  It should be also noted that the conformal balanced condition puts strong restrictions on Hermitian manifolds. For example, compact conformally balanced Hermitian manifolds $(M^{2n}, g, \omega)$ with a strong KT structure, $dH=0$, whose holonomy of the $\h\nabla$ connection is included in $SU(n)$, are Calabi-Yau and so $H=0$ \cite{SIGP3}.
\end{remark}

\subsubsection{Symmetries and the invariance of the dilaton}\label{sec:ex}

To illustrate what kind  of geometries   satisfy the condition (\ref{phig}), $D^i\big(e^{-2\Phi} K^\flat_i\big)=0$,   of the Theorem \ref{th:two}, we shall demonstrate the following proposition.
\begin{prop}\label{prop:kinv}
Let $(M^{2n}, g, \omega)$ be a compact Hermitian manifold such that it satisfies the scale invariance condition (\ref{sinv}) and suppose that $V^\flat=X^\flat+2d\Phi$ is Killing and
leaves both $H$ and  $P$ invariant, $\mathcal{L}_V H=\mathcal{L}_V P=0$. Then, $\Phi$ is also invariant, $\mathcal{L}_V \Phi=0$, and such a geometry satisfies (\ref{phig}).

Furthermore, if $(M^{2n}, g, \omega)$ is CYT, then the statement above holds for $K^\flat=\theta+2d\Phi$ provided again that $K$ is Killing and leaves both $H$ and $P$ invariant.
 \end{prop}
 \begin{proof}
The technique that we use  to prove  this statement is similar to that employed for the proof of Lemma 2.9 in \cite{as3}, where $P=0$. Taking the Lie derivative of (\ref{sinv}) along $K$ and using that $V$ is Killing and leaves both $H$ and $P$ invariant, one finds that
 \be
\h\nabla_i [V, X]^\flat_j=0\Longrightarrow \h\nabla_i \mathcal{L}_V (D_j\Phi)\Longrightarrow  D_i D_j (V\Phi)=0 \Longrightarrow D^2 (V\Phi)=0\Longrightarrow V\Phi=0~.
 \ee
 The last step follows because the $D^2 (V\Phi)=0$ condition implies that $V\Phi$ is constant. However, this constant must vanish because
 \be
 \int_{M^{2n}} d^{2n} x \,\,\sqrt{g}\,\, V\Phi=0~.
 \ee
 As $V$ is Killing and leaves invariant $\Phi$, (\ref{phig}) is satisfied.

 CYT manifolds are KT manifolds $(M^{2n}, g, \omega)$ such that the holonomy of $\h\nabla$ is included in $SU(n)$. Such manifolds satisfy the scale invariance condition (\ref{sinv}). To prove this, consider the Bianchi identity
\be
\h R_{ikmn}+\mathrm{cyclic ~in}~~ (k,\,m,\,n)=-\h \nabla_i H_{kmn}+\frac{1}{2} dH_{ikmn}~.
\ee
Contracting this with $I^m{}_j \omega^{mn}$ and upon using that the holonomy of $\h\nabla$ is included in $SU(n)$, we find that
\be
 \h R_{ij}-\frac{1}{4} P_{ij}=\h\nabla_i\theta_j~,
 \label{rhoeqnxx}
 \ee
 where $P_{ij}= dH_{ikmn} I^k{}_j\omega^{mn}$.
 Comparing (\ref{sinv}) with (\ref{rhoeqnxx}), we conclude that $X^\flat=\theta$.  The rest of the proof is as in the KT case above.
 \end{proof}

 \begin{remark}
 This proposition holds under weaker assumptions, see also \cite{as3}.  Consider a compact Riemannian manifold $(M^d, g, H)$  that satisfies the scale invariance condition (\ref{sinv}) with metric $g$ and torsion 3-form $H$.  Suppose that $K^\flat=Z^\flat+2 d\Phi$ is Killing and leaves both $H$ and $P$ invariant, then it also satisfies (\ref{phig}). The proof is along similar lines to the one explained above for Hermitian manifolds.
 \end{remark}

 \begin{remark}
 The proposition above can be further refined.  For example if $(M^{2n}, g, \omega)$ is a strong KT manifold that satisfies (\ref{sinv}), then it is a consequence of the results of \cite{GFS, GFJS} and \cite{pw} that it also satisfies (\ref{cinv}).  In such a case $K$ is $\h\nabla$-covariantly constant. Therefore, $K$ is both Killing and leaves $H$ invariant.  Such refinements of Proposition \ref{prop:kinv} will be explored in more detail in section   \ref{sec:k}. Explicit examples include compact group manifolds with a KT structure  and equipped with a bi-invariant metric and torsion 3-form \cite{Spindel, OP}.
 \end{remark}

\section{Symmetries of the moduli spaces}

To investigate the symmetries of the moduli space $\mathscr{M}_\HE(M^{2n})$ induced from those on $(M^{2n}, g, \omega)$, we shall need the proof of three key Lemmas \ref{le:one}, \ref{le:hor} and \ref{le:main}. Lemma \ref{le:one} holds for any KT manifold $(M^{2n}, g, \omega)$.  To prove \ref{le:hor} and \ref{le:main} depends on the invertibility of the operator $\mathcal{O}$. This holds whenever $M^{2n}$ is compact and the metric $g$ is either Gauduchon, $D^i\theta_i=0$, or satisfies the condition
$D^i(e^{-2\Phi} K^\flat_i)=0$ with $K^\flat=\theta+2d\Phi$.  Thus for the lemmas to hold, it is not necessary to assume that $g$ is Gauduchon.

Before we proceed to investigate the symmetries of the moduli space $\mathscr{M}_\HE(M^{2n})$ induced from those of the underlying manifold $(M^{2n}, g, \omega)$ in the presence of dilaton,
it is convenient first to carry out the computations in the Gauduchon gauge.  This will complement and provide a proof for the results originally stated in \cite{gp}.  Later in section \ref{sec:d}, we shall introduce a dilaton field on $M^{2n}$ and describe the symmetries of $\mathscr{M}_\HE(M^{2n})$ in the presence of a dilaton. For clarity, we shall mention explicitly whenever $g$ is restricted to satisfy the Gauduchon condition.

\subsection{Vector fields on the moduli space}

Given a vector field $X$ on $M^{2n}$, one can induce a vector field $a_X\equiv \ii_X F$
on $\mathscr{A}$, where $F$ is the curvature of the connection $A$. This can be generalised somewhat by noticing that $a_X\equiv \ii_X F-d_A\eta$ is also a vector field on $\mathscr{A}$ for every $\eta\in \Omega^0(P\times_{\Adj} \mathfrak{g})$. For Hermitian manifolds, this result admits a refinement as follows.

\begin{lemma}\label{le:one}
Let $(M^{2n}, g, \omega)$ be a Hermitian manifold and $F$ be the curvature of an Einstein-Hermitian connection $A$. If $X$ is a holomorphic and Killing vector field on $M^{2n}$, then
\be
 a_X=\ii_X F-d_A\eta
 \ee
  is a vector field on $\mathscr{A}^*_{\HE}$ for every $\eta\in \Omega^0(P\times_{\Adj} \mathfrak{g})$.
  \end{lemma}
  \begin{proof}
  It suffices to show that $d_Aa_X$ is a $(1,1)$-form and $\omega$-traceless, i.e. it satisfies the condition (\ref{tangentcon}). First,
  \be
  d_A a_X=d_A\ii_X F-d_A^2 \eta=(d_A\ii_X+\ii_X d_A) F-[F,\eta]_{\mathfrak{g}}=\mathcal{L}^A_X F-[F,\eta]_{\mathfrak{g}}~,
  \ee
  where the Bianchi identity $d_A F=0$ has been used. Using that $X$ is holomorphic, $\mathcal{L}_X I=0$,
  \begin{align}
   d_A a_X(IY, IZ)&=(\mathcal{L}^A_X F)(IY, IZ)-[F(IY, IZ),\eta]_{\mathfrak{g}}=(\mathcal{L}^A_X  I F)(Y, Z)-[F(IY, IZ),\eta]_{\mathfrak{g}}
   \cr
   &
   =(\mathcal{L}^A_X   F)(Y, Z)-[F(Y, Z),\eta]_{\mathfrak{g}}=d_A a_X(Y, Z)~,
   \end{align}
   where $(I F)(Y,Z)= F(IY, IZ)$, i.e. the condition that $F$ is a $(1,1)$-form can be expressed as $IF=F$, and we have used that $F$ is a $(1,1)$-form. Moreover, $\mathcal{L}^A_X\equiv \ii_X d_A+d_A \ii_X$.  Thus $d_A a_X$ is a $(1,1)$-form.

   Next, $d_A a_X$ is $\omega$-traceless as
   \be
   \omega\llcorner d_A a_X=\omega\llcorner \mathcal{L}^A_X F-[\omega\llcorner F,\eta]_{\mathfrak{g}}= \mathcal{L}^A_X \omega\llcorner F-[\omega\llcorner F,\eta]_{\mathfrak{g}}=0
   \ee
   as $\omega\llcorner F$ is constant that commutes with the other elements of $\mathfrak{g}$. We have also used that $X$ is holomorphic and Killing and so
   $\mathcal{L}_X\omega=0$. This establishes the statement.
  \end{proof}

For $a_X=\ii_X F-d_A\eta$ to be the horizontal lift of a vector field $\alpha_X$ on the moduli space $\mathscr{M}^*_\HE$, it has to satisfy the horizontality condition as well. This can be achieved by imposing a condition on $\eta$ as that  so far has remained unrestricted.

\begin{lemma}\label{le:hor}
Let $(M^{2n}, g, \omega)$ be a Hermitian manifold with a Killing and holomorphic vector field $X$, $\mathcal{O}$ be an invertible operator and $F$ be the curvature of an Einstein-Hermitian connection $A$.  The vector field $a_X=\ii_X F-d_A\eta$ is tangent to $\mathscr{A}^*_\HE(M^{2n})$ and horizontal provided that
\be
\mathcal{O}\eta=-F_{ij}\h\nabla^i X^j~.
\label{oeta}
\ee
\end{lemma}
\begin{proof}
We have already demonstrated in the previous lemma that $a_X$ is tangent to $\mathscr{A}^*_\HE(M^{2n})$. Assuming that $a_X$ is horizontal, $a_X=a^h_X$, we impose the horizontality condition on $a_X$ to find that
\be
D^i_A\ii_X F_i+\theta^i \ii_X F_i-\mathcal{O}\eta=0~.
\label{hhcon}
\ee
However, the Bianchi identity for $F$ and the Hermitian-Einstein condition on $F$ imply that
\be
\h\nabla_A^iF_{ij}+\theta^i F_{ij}=0~.
\ee
Moreover
\be
D^i_A\ii_X F_i= \h\nabla_A^i(X^j F_{ji})=\h\nabla^i X^j F_{ji}+\nabla_A^i F_{ji} X^j=-\h\nabla^i X^j F_{ij}-F_{ji} X^j \theta^i
\ee
Substituting this in (\ref{hhcon}), one finds (\ref{oeta}).  As $\mathcal{O}$ is invertible, there is always an $\eta$ given $X$ and $F$ such that $a_X$ satisfies the horizontality   condition.  Therefore, $a_X$ is the horizontal lift of a vector field $\alpha_X$ on the moduli space $\mathscr{M}^*_\HE(M^{2n})$.
\end{proof}

\begin{remark}
It is clear from the previous lemma that if $X$ is in addition $\h\nabla$-covariantly constant, then $\eta=0$ and $\ii_X F$ is horizontal. This reproduces the result that has already been established in \cite{gp}. Note also that $\mathcal{O}$ is an invertible operator provided that $M^{2n}$ is compact and $g$ either satisfies the Gauduchon condition or  satisfies the condition
$D^i(e^{-2\Phi} K^\flat_i)=0$ with $K^\flat=\theta+2d\Phi$.
\end{remark}

\subsection{Holomorphic and Killing vector fields on the moduli space}

To proceed with the investigation of symmetries of the moduli space $\mathscr{M}^*_\HE(M^{2n})$, we have to evaluate the components of the curvature $\Theta$ associated with the horizontality condition (\ref{hor}) along the vector field $\alpha_X$ of the moduli space. Recall that $\Theta$ is uniquely  defined by the equation
(\ref{diftheta}) as the operator ${\mathcal O}$ is invertible.  We shall use this to prove the following key lemma.

\begin{lemma}\label{le:main}
Let $X$ be holomorphic and Killing vector field on $M^{2n}$, $\mathcal{O}$ be an invertible operator and $F$ be the curvature of an Einstein-Hermitian connection. Consider $a_X^h=\ii_X F-d_A\eta$ with $\mathcal{O}\eta=-F_{ij}\h\nabla^i X^j$, then
\be
\Theta(a^h, a_X^h)=-\Theta(a_X^h, a^h)=\ii_X a^h+ a^h\cdot \eta~,
\label{thaax}
 \ee
 for any other horizontal vector field $a^h$.
 \end{lemma}
\begin{proof}
 As $a_X^h$ satisfies the horizontality condition, we have that
 \be
 D_A^i (\ii_X F-d_A\eta)_i+ \theta^i (\ii_X F-d_A\eta)_i=0~.
 \label{horcon2}
 \ee
 To continue, we shall evaluate the directional derivative of the horizontality condition above along the tangent vector $a$.  For this, we consider first
 \be
a\cdot (D_A^i \ii_X F_i+ \theta^i\ii_X F_i)= [a^j, X^i F_{ij}]_{\mathfrak{g}}+ D_A^j(X^i (d_A a)_{ij})+ \theta^j (d_A a)_{ij} X^i~.
\label{com1}
\ee
Next, we notice that
\be
\theta^j (d_A a)_{ij} X^i= X^i D_{Ai} (\theta^j a_j)-\theta^j D_{Aj}(a_i X^i)~,
\ee
where we have used that $\mathcal{L}_X\theta=0$ as $\theta$ depends on the metric and complex structure of $M^{2n}$ and $X$ is holomorphic and Killing.

Then, we find that
\begin{align}
D_A^j(X^i (d_A a)_{ij})&=X^i D_{Aj} D_{Ai} a^j- D_A^2 (a_i X^i)+ a_i D^2 X^i
\cr
&
=  X^i D_{Ai} D_{Aj} a^j +X^i R_{ij} a^j+ [a^j, F_{ij} X^i]_{\mathfrak{g}}- D_A^2 (a_i X^i)- a^i R_{ij} X^j
\cr
&
= X^i D_{Ai} D_{Aj} a^j+[a^j, F_{ij} X^i]_{\mathfrak{g}}-  D_A^2 (a_i X^i)~,
\end{align}
where we have used that $X$ is Killing and so $D^2 X_i= -R_{ij} X^j$. Putting these terms in (\ref{com1}), we find that
\be
a\cdot (D_A^i \ii_X F_i+ \theta^i\ii_X F_i)=-D_A^2 (a_i X^i)-\theta^j D_{Aj} (a_i X^i)+ 2 [a^j, X^i F_{ij}]_{\mathfrak{g}}+ X^i D_{Ai} (D_{Aj} a^j+\theta^j a_j)~.
\label{hconxf}
\ee
If the tangent vector $a$ satisfies the horizontality condition, $D_{Aj} a^j+\theta^j a_j=0$,  then
\be
a^h\cdot (D_A^i \ii_X F_i+ \theta^i\ii_X F_i)=-D_A^2 (a^h_i X^i)-\theta^j D_{Aj} (a^h_i X^i)+ 2 [a^{hj}, X^i F_{ij}]_{\mathfrak{g}}~.
\label{hconxf2}
\ee
Moreover, we also have
\be
 -a\cdot (D_A^2 \eta+ \theta^i D_{Ai}\eta)=- D_A^2 (a\cdot \eta)- \theta^i D_{Ai} (a\cdot \eta)-2 [a^i, D_{Ai}\eta]_{\mathfrak{g}}-[D_A^ia_i+\theta^i a_i, \eta]_{\mathfrak{g}}~.
 \label{heta}
 \ee
 Imposing the horizontality condition on $a$, we deduce that
 \be
 -a^h\cdot (D_A^2 \eta+ \theta^i D_{Ai}\eta)=- D_A^2 (a^h\cdot \eta)- \theta^i D_{Ai} (a^h\cdot \eta)-2 [a^{hi}, D_{Ai}\eta]_{\mathfrak{g}}~.
 \label{heta1}
 \ee

 To evaluate the directional derivative of (\ref{horcon2}) along the horizontal vector field $a^h$, it remains to add (\ref{hconxf2}) and (\ref{heta1}) and set the result to zero.  This yields
 \be
 \mathcal{O} (\ii_{ X}  a^h+  a^h\cdot \eta)= 2[ a^{hi},  X^jF_{ji}- D_{Ai}\eta]_{\mathfrak{g}}=2[a^{hi},  a^h_{{ X}i}]_{\mathfrak{g}}~.
 \ee
Comparing this to the equation (\ref{diftheta}) that determines $\Theta$, we  arrive at  (\ref{thaax}) as the operator $\mathcal{O}$ is invertible. This is always the case provided that $M^{2n}$ is compact and either $g$ is Gauduchon or satisfies $D^i(e^{-2\Phi} K^\flat_i)=0$ with $K^\flat=\theta+2d\Phi$.
\end{proof}

\begin{theorem}\label{th:one}
Let  $X$ be a holomorphic and Killing vector field on the compact Hermitian manifold $(M^{2n}, g, \omega)$ and  $g$ be the Gauduchon metric. Then, $\alpha_X$ is a holomorphic and Killing vector field on $\mathscr{M}^*_{\HE}$.
\end{theorem}
\begin{proof}

Let us begin with the proof of the statement that $\alpha_X$ is Killing. To do this, we use the formula
\begin{align}
\mathcal{L}_{\alpha_{{}_X}} ( \alpha_1,& \alpha_2)_{{\mathscr{M}}^*_{\HE}}=\alpha_{{}_X} \cdot ( \alpha_1, \alpha_2)_{{\mathscr{M}}^*_{\HE}}-(\lsq \alpha_{{}_X}, \alpha_1\rsq, \alpha_2)_{{\mathscr{M}}^*_{\HE}}-( \alpha_1, \lsq \alpha_{{}_X}, \alpha_2\rsq)_{{\mathscr{M}}^*_{\HE}}
\cr
&
=a_X^h \cdot ( a^h_1, a^h_2)_{{\mathscr{A}}^*_{\HE}}-(\lsq a_X^h, a^h_1\rsq^h, a^h_2)_{{\mathscr{A}}^*_{\HE}}-( a^h_1, \lsq a_X^h, a^h_2\rsq^h)_{{\mathscr{A}}^*_{\HE}}~,
\end{align}
for the Lie derivative of the metric $( \cdot, \cdot)_{{\mathscr{M}}^*_{\HE}}$ on the moduli space $\mathscr{M}^*_{\HE}$ -- the first line is the standard formula of the Lie derivative of a $(0,2)$-tensor along a vector field.

Using that $a_X^h=\ii_XF-d_A\eta$, which follows  from the results of Lemma \ref{le:hor}, and the definition of the inner product in (\ref{innerform}), the above equation gives
\begin{align}
\mathcal{L}_{\alpha_{{}_X}} ( \alpha_1\,,&\, \alpha_2)_{{\mathscr{M}}^*_{\HE}}=\int_{M^{2n}} d^{2n}x \sqrt{g} g^{-1} \Big((\ii_X F-d_A\eta)\cdot\langle a_1^h\,,\, a_2^h\rangle_{\mathfrak{g}}-\langle\lsq \ii_XF-d_A\eta, a_1^h\rsq^h\,,\, a_2^h\rangle_{\mathfrak{g}}
\cr
& \qquad\qquad \qquad
-\langle a_1^h, \lsq\ii_XF-d_A\eta \,,\, a_2^h\rsq^h\rangle_{\mathfrak{g}}\Big)
\cr
&
=\int_{M^{2n}} d^{2n}x \sqrt{g} g^{-1} \Big(\langle(\ii_X F-d_A\eta)\cdot a_1^h\,,\, a_2^h\rangle_{\mathfrak{g}}+\langle a_1^h\,,\, (\ii_X F-d_A\eta)\cdot a_2^h\rangle_{\mathfrak{g}}
\cr
&\qquad\qquad-\langle\lsq \ii_XF-d_A\eta, a_1^h\rsq^h\,,\, a_2^h\rangle_{\mathfrak{g}}-
\langle a_1^h, \lsq\ii_XF-d_A\eta, a_2^h\rsq^h\rangle_{\mathfrak{g}}\Big)
\cr
&
=\int_{M^{2n}} d^{2n}x \sqrt{g} g^{-1} \Big(\langle\lsq\ii_X F-d_A\eta\,,\, a_1^h\rsq+ d_A \ii_Xa_1^h-d_A (a_1^h\cdot \eta)\,,\, a_2^h\rangle_{\mathfrak{g}}
\cr
&\qquad\qquad
+\langle a_1^h\,,\, \lsq\ii_X F-d_A\eta\,,\, a_2^h\rsq+ d_A \ii_Xa_2^h-d_A (a_2^h\cdot \eta)\rangle_{\mathfrak{g}}
\cr
&\qquad\qquad
-\langle\lsq \ii_X F-d_A\eta\,,\, a_1^h\rsq+d_A \Theta(\ii_XF-d_A\eta, a_1^h)\,,\, a_2^h\rangle_{\mathfrak{g}}
\cr
&\qquad\qquad
-
\langle a_1^h\,,\, \lsq \ii_XF-d_A\eta, a_2^h\rsq+d_A \Theta(\ii_XF-d_A\eta, a_2^h)\rangle_{\mathfrak{g}}\Big)
\cr
&
=\int_{M^{2n}} d^{2n}x \sqrt{g} g^{-1} \Big(\langle (\ii_X d_A+d_A \ii_X) a_1^h\,,\, a_2^h\rangle_{\mathfrak{g}}+\langle a_1^h\,,\,  (\ii_X d_A+d_A \ii_X) a_2^h\rangle_{\mathfrak{g}}\Big)
\cr
&
=\int_{M^{2n}} d^{2n}x \sqrt{g} g^{-1}  \mathcal{L}_X \langle a_1^h\,,\, a_2^h\rangle_{\mathfrak{g}}=0~,
\label{kproof}
\end{align}
as $X$ is a Killing vector field, where we have used the result of the main Lemma \ref{le:main} and in particular (\ref{thaax}). Furthermore to perform the above calculation, we have used that
\begin{align}
(\ii_X F-d_A\eta)\cdot a_1^h&=\lsq \ii_X F-d_A\eta, a_1^h\rsq+a_1^h\cdot (\ii_X F-d_A\eta)
\cr
&=\lsq \ii_X F-d_A\eta, a_1^h\rsq+\ii_X d_A a_1^h- [a_1^h, \eta]_{\mathfrak{g}}-d_A (a_1^h \cdot\eta)
\end{align}
and similarly for $(\ii_X F-d_A\eta)\cdot a_2^h$.  The terms involving $[a_1^h, \eta]_{\mathfrak{g}}$ and $[a_2^h, \eta]_{\mathfrak{g}}$ cancel in
(\ref{kproof}) using that $\langle \cdot, \cdot \rangle_{\mathfrak{g}}$ is bi-invariant.
It is clear from (\ref{kproof}) that $\alpha_{{}_X}$ is a Killing vector field on the moduli space ${\mathscr{M}}^*_{\HE}$.

Next, we shall demonstrate that $\alpha_X$ is holomorphic. As $\alpha_X$ is Killing, it suffices to prove that the Hermitian form $\Omega_{\mathscr{M}^*_{\HE}}$ is invariant under the action of $\alpha_X$. For this, let us first calculate the exterior derivative of $\Omega_{\mathscr{M}^*_{\HE}}$.  This has already been evaluated in \cite{lubke}, see also \cite{gp},  but we state here the computation for completeness. Using that $\Omega_{\mathscr{A}^*_{\HE}}$ is closed 2-form as it is the restriction of the Hermitian form $\Omega_{\mathscr{A}}$ onto ${\mathscr{A}_{\HE}}$, we deploy the standard expression for the exterior derivative of a 2-form to find
\begin{align}
d\Omega_{\mathcal{M}^*_{\HE}}&(\alpha_1, \alpha_2, \alpha_3)=\Big(\alpha_1\cdot \Omega_{\mathcal{M}^*_{\HE}}(\alpha_2, \alpha_3)-\Omega_{\mathcal{M}^*_{\HE}}(\lsq \alpha_1, \alpha_2\rsq, \alpha_3)\Big)+\mathrm{cyclic~ in~} (\alpha_1, \alpha_2, \alpha_3)
 \cr
 &
 =\Big(a^h_1 \cdot \Omega_{\mathcal{A}^*_{\HE}}(a_2^h, a_3^h)-\Omega_{\mathcal{A}^*_{\HE}}(\lsq a^h_1, a_2^h\rsq^h, a_3^h)\Big)+ \mathrm{cyclic~ in~} (a^h_1, a^h_2, a^h_3)
\cr
&=\Omega_{\mathcal{A}^*_{\HE}}(\lsq a^h_1, a_2^h\rsq^v, a_3^h)+ \mathrm{cyclic~ in~} (a^h_1, a^h_2, a^h_3)
\cr
&=\frac{1}{(n-1)!} \,\int_{M^{2n}} \omega^{n-1}\wedge \langle \lsq a^h_1, a_2^h\rsq^v\wedge  a^h_3\rangle_{\mathfrak{g}}+ \mathrm{cyclic~ in~} (a^h_1, a^h_2, a^h_3)
\cr
&=-\frac{1}{(n-1)!} \,\int_{M^{2n}} \omega^{n-1}\wedge \langle d_A\Theta (a_1^h, a_2^h) \wedge  a^h_3\rangle_{\mathfrak{g}}+ \mathrm{cyclic~ in~} (a^h_1, a^h_2, a^h_3)
\cr
&=\frac{1}{(n-1)!} \,\Big(\int_{M^{2n}} d\omega^{n-1}\wedge \langle \Theta (a_1^h, a_2^h),   a^h_3\rangle_{\mathfrak{g}}
\cr &\qquad\qquad+  \,\int_{M^{2n}} \omega^{n-1}\wedge \langle \Theta (a_1^h, a_2^h),   d_Aa^h_3\rangle_{\mathfrak{g}}\Big) + \mathrm{cyclic~ in~} (a^h_1, a^h_2, a^h_3)
\cr
&=\frac{1}{(n-1)!} \,\int_{M^{2n}} d\omega^{n-1}\wedge \langle \Theta (a_1^h, a_2^h),   a^h_3\rangle_{\mathfrak{g}} + \mathrm{cyclic~ in~} (a^h_1, a^h_2, a^h_3)
\label{domegab}
\end{align}
where $\lsq \cdot, \cdot\rsq^v$ is the vertical component of the commutator and so $\lsq \cdot, \cdot\rsq^v=-d_A \Theta(\cdot, \cdot)$. We have also used that $\omega^{n-1}\wedge d_A a_1^h=0$ as $a_1^h$ is tangent to
${\mathscr{A}_{\HE}}$ and similarly for $a_2^h$ and $a_3^h$.

To calculate the Lie derivative of $\Omega_{\mathcal{M}^*_{\HE}}$ along $\alpha_X$, we use the well-known expression of the Lie derivative of a 2-form along a vector field to find
\begin{align}
\mathcal{L}_{\alpha_{{}_X}}& \Omega(\alpha_1, \alpha_2)=\ii_{\alpha_{{}_X}} d\Omega(\alpha_1, \alpha_2)+d \ii_{\alpha_{{}_X}} \Omega(\alpha_1, \alpha_2)
\cr
&
=\frac{1}{(n-1)!} \int_{M^{2n}} \omega^{n-1}\wedge \Big(-\langle d_A\Theta (\ii_X F-d_A \eta, a_1^h)\wedge a_2^h\rangle_{\mathfrak{g}}-\langle d_A\Theta ( a_2^h, \ii_X F-d_A\eta)\wedge a_1^h\rangle_{\mathfrak{g}}
\cr
&\qquad -\langle d_A\Theta ( a_1^h, a_2^h)\wedge (\ii_X F-d_A\eta)\rangle_{\mathfrak{g}}+ a_1^h \cdot \langle (\ii_XF-d_A\eta)\wedge a_2^h\rangle_{\mathfrak{g}}
\cr
&\qquad
- a_2^h \cdot \langle (\ii_XF-d_A\eta)\wedge a_1^h\rangle_{\mathfrak{g}}-\langle (\ii_XF-d_A\eta)\wedge \lsq a_1^h, a_2^h\rsq^h \rangle_{\mathfrak{g}}\Big)
\cr
&
=\frac{1}{(n-1)!} \int_{M^{2n}} \omega^{n-1}\wedge \Big(\langle d_A(\ii_Xa_1^h+a_1^h\cdot \eta)\wedge a_2^h\rangle_{\mathfrak{g}}-\langle d_A (\ii_X a_2^h+ a_2^h\cdot\eta)\wedge a_1^h\rangle_{\mathfrak{g}}
\cr
&\qquad
-\langle d_A\Theta ( a_1^h, a_2^h)\wedge (\ii_X F-d_A\eta)\rangle_{\mathfrak{g}}+   \langle (\ii_Xd_A a_1^h-d_A(a_1^h\cdot \eta))\wedge a_2^h\rangle_{\mathfrak{g}}
\cr
&\qquad -  \langle (\ii_Xd_A a_2^h-d_A(a_2^h\cdot \eta)) \wedge a_1^h\rangle_{\mathfrak{g}}
-\langle (\ii_XF-d_A\eta)\wedge d_A\Theta(a_1^h, a_2^h)  \rangle_{\mathfrak{g}}\Big)
\cr
&
=\frac{1}{(n-1)!} \int_{M^{2n}} \omega^{n-1}\wedge \big(\langle \mathcal{L}_X a_1^h\wedge a^h_2\rangle_{\mathfrak{g}}+\langle  a_1^h\wedge \mathcal{L}_X a^h_2\rangle_{\mathfrak{g}}\big)=0~,
\label{holx}
\end{align}
 because $X$ is holomorphic and Killing  and so  $\mathcal{L}_X\omega=0$. Apart from repetitively using the main Lemma \ref{le:main} and in particular (\ref{thaax}), we have also utilised the identity
\be
 a_1^h \cdot  (\ii_XF-d_A\eta)=\ii_X d_A a_1^h-d_A (a_1^h\cdot \eta)-[a_1^h, \eta]_{\mathfrak{g}}~,
\ee
 and similarly for $a_2^h \cdot  (\ii_XF-d_A\eta)$.  The terms $[a_1^h, \eta]_{\mathfrak{g}}$ and $[a_2^h, \eta]_{\mathfrak{g}}$ that arise as a result of the computation above do not appear in the calculation described in (\ref{holx}) because they cancel as a consequence of the bi-invariance of the inner product $\langle\cdot, \cdot\rangle_{\mathfrak{g}}$.
 Thus, we have shown that $\alpha_{{}_X}$ is also holomorphic and this concludes the proof of the theorem.
\end{proof}

\begin{remark}
Given two holomorphic and Killing vector fields $X$ and $Y$ on $M^{2n}$ and the associated holomorphic and Killing vector fields $\alpha_X$ and $\alpha_Y$ on $\mathcal{M}^*_{\HE}$, a short computation reveals that the commutator of the latter is
\be
\lsq \alpha_X, \alpha_Y\rsq\equiv \lsq a^h_X,  a^h_Y\rsq^h=-\alpha_{[X,Y]}~.
\label{xycom}
\ee
Therefore, the commutator of $\alpha_X$ and $\alpha_Y$ is determined in terms of that of $X$ and $Y$.

Given a holomorphic vector field $X$ on $M^{2n}$, there is a second one given by $Z=-IX$.  Moreover, $Z$ is Killing provided that $X$ is Killing and $d X^\flat$ is a $(1,1)$-form on $M^{2n}$. Furthermore, if $Z$ is holomorphic and Killing, then $[X, Z]=0$ and so $\lsq \alpha_X, \alpha_Z\rsq=0$.
\end{remark}

\begin{remark}
It has been shown in \cite{gp} that if $X$ is holomorphic and $\h\nabla$-covariantly constant -- the latter condition implies that $X$ is Killing -- and
$X^\flat\wedge \theta$ is a $(1,1)$-form, then $\alpha_X$ is $\h{\mathcal D}$-covariantly constant, where $\h{\mathcal D}$ is the compatible with the Hermitian structure connection with skew-symmetric torsion on $\mathscr{M}^*_{\HE}$.
\end{remark}

\section{Symmetries of moduli spaces  and the dilaton field}\label{sec:d}

%\subsection{Symmetries of moduli spaces}

Symmetries of the underlying manifold in the presence of a dilaton can be extended to the moduli space of Hermitian-Einstein connections. The proof of such results is similar to those that are given for metrics that obey the Gauduchon condition in the previous section.  Because of this we shall state the results and only emphasise the few differences that appear in the proof of the statements.

%\begin{theorem}\label{th:holkill}
%Let $(M^{2n}, g, \omega)$ be a compact KT manifold with dilaton $\Phi$. Suppose that the condition (\ref{phig}), $D^i\big(e^{-2\Phi} K^\flat_i\big)=0$ with $K^\flat=\theta+2d\Phi$,  is satisfied and moreover $M^{2n}$ admits a holomorphic and Killing vector field $X$ such that $X\Phi=0$, then $\alpha_X$ is holomorphic and Killing vector field on $\mathscr{M}^*_\HE$.
%\end{theorem}

\vskip 0.3cm
{\bf Theorem} \ref {th:holkill} stated at the introduction.
\begin{proof}
The key observation is that both the construction of the induced the vector $\alpha_X$ on the moduli space $\mathscr{M}^*_\HE$, in particular its horizontal lift $a^h_X$ on $\mathscr{A}^*_\HE$ described in Lemma \ref{le:hor}, and the main Lemma \ref{le:main} apply unaltered in the presence of a dilaton $\Phi$ provided that (\ref{phig}), $D^i(e^{-2\Phi} K^\flat_i)=0$,  holds. As a result to prove that $\alpha_X$ is Killing, after performing the same steps in computation as those in (\ref{kproof}) but now for the metric on the moduli space $( \cdot, \cdot)_{{\mathscr{M}}^*_{\HE}}$ that depends on the dilaton $\Phi$, (\ref{Aform2}), the Lie derivative of this metric  reads
\be
\mathcal{L}_{\alpha_{{}_X}} ( \alpha_1, \alpha_2)_{{\mathscr{M}}^*_{\HE}}=\int_{M^{2n}} d^{2n}x \sqrt{g}\,\, e^{-2\Phi}\,\, g^{-1}  \mathcal{L}_X \langle a_1^h\,,\, a_2^h\rangle_{\mathfrak{g}}=0~,
\ee
where in the last step we have used that $X$ is Killing and $X\Phi=0$.

To prove that $\alpha_X$ is holomorphic,  we  perform similar steps  to those described in the computation (\ref{holx}).  As a result the Lie derivative of the Hermitian form on the moduli space, which now depends on $\Phi$,  along $\alpha_X$ is
\be
\mathcal{L}_{\alpha_{{}_X}} \Omega(\alpha_1, \alpha_2)=\frac{1}{(n-1)!} \int_{M^{2n}} \,\, e^{-2\Phi}\,\,\omega^{n-1}\wedge \mathcal{L}_X \langle  a_1^h\wedge a^h_2\rangle_{\mathfrak{g}}=0~,
\label{holx2}
\ee
where in the last step, we have used that $X\Phi=0$ and that $X$ is holomorphic and Killing and so $\mathcal{L}_X \omega=0$.  Thus we have established that $\alpha_X$ is both Killing and holomorphic that concludes the proof in the presence of a dilaton $\Phi$ \end{proof}

It remains to investigate the special case that $X$ is $\h\nabla$-covariantly constant and holomorphic.  The former conditions implies that $X$ is Killing and $dX^\flat=\ii_XH$.  Conversely, if $X$ is Killing and $dX^\flat=\ii_XH$, then $X$ is $\h\nabla$-covariantly constant.  This will be used in the theorem below to demonstrate that $\alpha_X$ is $\h{\mathcal{D}}$-covariantly constant on $\mathscr{M}^*_\HE$.   Assumming that $(M^{2n}, g, \omega)$ is a compact KT manifold with dilaton $\Phi$ such that the condition (\ref{phig}), $D^i\big(e^{-2\Phi} K^\flat_i\big)=0$ with $K^\flat=\theta+2d\Phi$,  is satisfied, one can establish the following result.

%\begin{theorem}\label{th:phihnabla}
%Let $X$ be $\h\nabla$-parallel and holomorphic on $M^{2n}$, and $X\Phi=Y\Phi=0$ with $Y=-IX$. Then,  $\alpha_X$ is $\h{\mathcal D}$-covariantly constant on $\mathscr{M}^*_{\HE}$  provided that $X^\flat\wedge K^\flat$ is a $(1,1)$-form on $M^{2n}$, where $K^\flat=\theta+2 d\Phi$.
%\end{theorem}
\vskip 0.3cm
{\bf Theorem} \ref{th:phihnabla} stated at the introduction.
\begin{proof}
 We have already demonstrated that $\alpha_X$ is Killing and holomorphic as a consequence of Theorem \ref{th:holkill}. It remains to show that $\alpha_X$ is $\h{\mathcal D}$-covariantly constant. For this, it suffices to prove that $d\alpha_X^\flat=\ii_{\alpha_X}\mathcal{H}$.

  To begin, it is a consequence of Lemma \ref{le:hor} and in particular (\ref{oeta}) that $a_X^h=\ii_XF$ as $\h\nabla X=0$.  To continue, we have to compute $\ii_{\alpha_X} \mathcal{H}$. Indeed,  a computation similar to that presented in \cite{gp} reveals that
\begin{align}
\mathcal {H}&(\alpha_{{}_X}, \alpha_2, \alpha_3)= \int_{M^{2n}} d^{2n}x \sqrt g\,\, e^{-2\Phi}\,\, g^{-1}\big( \langle D_{A} (\ii_X a_2^h), a_3^h\rangle_{\mathfrak{g}}-\langle D_{A} (\ii_X a_3^h), a_2^h\rangle_{\mathfrak{g}}-\langle d_A\Theta ( a^h_2, a_3^h),  \ii_X F \rangle_{\mathfrak{g}}\big)
\cr
&= \int_{M^{2n}}d^{2n}x \sqrt g\,\, e^{-2\Phi}\,\,  \big( (X\wedge K)^{ij} \langle a^h_{2i}, a^h_{3j}\rangle_{\mathfrak{g}}   -g^{-1} \langle d_A\Theta ( a^h_2, a_3^h),  \ii_X F \rangle_{\mathfrak{g}}\big)
\label{HF}
\end{align}
where $K^\flat=\theta+2d\Phi$.
On the other hand the associated 1-form $\mathcal{F}_X$ to the vector field $a^h_X=\ii_X F$ with respect to the metric $\mathcal{G}$ on the moduli space ${\mathscr{M}}^*_{\HE}$ is
\begin{align}
\alpha_{{}_X}^\flat(\alpha)\equiv \mathcal{F}_X(a^h)\equiv \int_{M^{2n}} \sqrt{g}\,\, e^{-2\Phi}\,\, g^{-1} \langle \ii_X F, a^h\rangle_{\mathfrak{g}}~.
\end{align}
Taking the exterior derivative of this yields
\begin{align}
d\mathcal{F}_X(a^h_2, a^h_3)=\int_{M^{2n}} d^{2n}x \sqrt{g}\,\, e^{-2\Phi}\,\, g^{-1}\big(   \langle \ii_X d_A a_2^h, a_3^h\rangle_{\mathfrak{g}}-  \langle \ii_X d_Aa_3^h, a_2^h\rangle_{\mathfrak{g}}- \langle \ii_X F, d_A\Theta(a_2^h, a_3^h) \rangle_{\mathfrak{g}}\big)~.
\label{interm}
\end{align}
To  evaluate the first two terms in the right hand side of the last  equation above, we use the identity
\begin{align}
&\langle\ii_X d_A a^h_{2i}, a^{hi}_3\rangle_{\mathfrak{g}}- (a^h_3, a^h_2)
=Y^i D_{i} \langle I^{jk} a^h_{2j}, a^h_{3k}\rangle_{\mathfrak{g}}-D_n \langle Y^i a^h_{2i}, I^n{}_j a^{j}h_3\rangle_{\mathfrak{g}} +D_n \langle Y^i a^h_{3i},  I^n{}_j a^{hj}_2 \rangle_{\mathfrak{g}}
\cr
&\qquad\qquad+ \langle Y^i a^h_{2i},  D_n I^n{}_j a^{hj}_3\rangle_{\mathfrak{g}}-\langle Y^i a^h_{3i},  D_n I^n{}_j a^{hj}_2\rangle_{\mathfrak{g}}
\label{anid}
\end{align}
derived in \cite{gp}. After substituting (\ref{anid}) into (\ref{interm}), we find that the contribution of the first term in the right hand side of (\ref{anid}) vanishes upon using that $Y$ is Killing and $Y\Phi=0$  -- it gives rise to a surface term while $M^{2n}$ is compact without boundary.   To continue, we integrate by parts  the contribution of the second and third terms in the right hand side of (\ref{anid}) in (\ref{interm}). This will result into terms that will depend on the first derivative of the dilaton $\Phi$.  Next,  we combine this result with the contribution of the fourth and fifth terms of (\ref{anid}) in (\ref{interm}). The latter terms can be expressed in terms of the Lee form $\theta$ of $M^{2n}$. Putting all these together, it yields that
\begin{align}
d\mathcal{F}_X(a^h_2, a^h_3)&=\int_{M^{2n}} d^{2n}x \sqrt{g}\,\, e^{-2\Phi}\,\, \big(  (IX\wedge IK)^{ij} \langle a_{2i}^h, a_{3j}^h\rangle_{\mathfrak{g}}- g^{-1} \langle \ii_X F, d_A\Theta(a_2^h, a_3^h) \rangle_{\mathfrak{g}}\big)~.
\label{dF}
\end{align}
 On comparing (\ref{dF}) with (\ref{HF}), we conclude that $ d\alpha^\flat_{{}_X}=\ii_{\alpha_{{}_X}} \mathcal{H}$
 provided that
\be
X^\flat\wedge K^\flat=\ii_I X^\flat\wedge \ii_I K^\flat~,
\label{xtii}
 \ee
which is one of the hypotheses of the theorem. As $\alpha_X$ is also Killing, this implies that $\alpha_X$ is $\h{\mathcal{D}}$-covariantly constant.
\end{proof}

\begin{remark}
The commutator of two Killing and holomorphic vector fields $\alpha_X$ and $\alpha_Y$ is given as in (\ref{xycom}) as it is not affected by the presence of a dilaton.
\end{remark}

\section{Holomorphic and Killing vector fields on geometries with a dilaton field}\label{sec:k}

We have demonstrated that under certain conditions holomorphic and Killing vector fields on the underlying Hermitian manifold $(M^{2n}, g, \omega)$ induce holomorphic and Killing vector fields
on the moduli space $\mathscr{M}_\HE^*(M^{2n})$ of Hermitian-Einstein connections. In the presence of a dilaton $\Phi$, the investigation of the geometry of $\mathscr{M}_\HE^*(M^{2n})$ requires the condition
\be
D^i \big(e^{-2\Phi} K^\flat_i\big)=0~,~~~K^\flat=\theta+2d\Phi~.
\label{ktheta}
\ee
It is clear that whether a geometry on $M^{2n}$ satisfies this condition depends on the properties of the 1-form $K^\flat$. Moreover, in the investigation of symmetries on the moduli space $\mathscr{M}_\HE^*(M^{2n})$ in section \ref{sec:d}, the vector field $K$ also plays key role.
All these warrant a systematic understanding of the properties of the vector field $K$ and this is the main topic that will be explored below.

In what follows in this section, we shall consider manifolds $(M^d, g, H)$  with metric $g$ and 3-form torsion $H$ such that
\be
\h R_{ij}+ 2 \h\nabla_i\partial_j\Phi-{1\over 4} P_{ij}=0~,~~~
\label{seqn}
\ee
for $P_{ij}=P_{ji}$ that will be specified below.  As we have already explained, these geometries include all those that have found applications in either string theory or in the theory of geometric (generalised) flows.
The condition on the geometry (\ref{seqn}) is  one of the assumptions motivated by  the applications mentioned above. We shall not assume that $M^d$ is compact unless it is explicitly stated.

\subsection{Symmetries of KT geometry with dilaton}

The main result here is an extension  of the proposition 4.1  proven by Streets and Ustinovskiy in \cite{jsyu}   to geometries  that satisfy (\ref{seqn}).  In particular, we shall show the following.

\begin{prop}\label{prop:hol}
Let $(M^{2n}, g, \omega)$ be a KT manifold that satisfies\footnote{It would be of interest to construct (smooth) compact KT manifolds that satisfy (\ref{seqn}) for which the (reduced)
holonomy of $\h\nabla$ is strictly $U(n)$, i.e. examples for which  (\ref{seqn}) is not a consequence of  reduction for the holonomy of  $\h\nabla$ to a proper subgroup of $U(n)$.}  (\ref{seqn}) with $P_{ij}=dH_{ikmn} I^k{}_j \omega^{mn}$. The following statements hold
\be \mathcal {L}_K g^{2,0}=0~,\ee
\be \mathcal {L}_K \omega^{2,0}=0~,\ee
\be\mathcal {L}_K \omega= -2 \rho^{1,1}~,\ee
\be \mathcal{L}_{Y} g=\mathcal{L}_{Y}\omega=0, ~\mathrm{ where}~~ Y=-IK~,
\ee
where $K^\flat=\theta+2 d\Phi$ and $\rho$ is the Ricci form of $\h\nabla$ connection.

In particular, $K$ is holomorphic, but not necessarily Killing, vector field and $Y=-IK$ is both holomorphic and Killing vector field.
\end{prop}
\begin{proof}

To prove the first statement, first notice that $P_{ij}=P_{ji}$ as $dH$ is a (2,2)-form. Next consider the Bianchi identity
\be
\h R_{ikmn}+\mathrm{cyclic ~in}~~ (k,\,m,\,n)=-\h \nabla_i H_{kmn}+\frac{1}{2} dH_{ikmn}~.
\label{bian2}
\ee
Contracting this with $I^m{}_j \omega^{mn}$ and upon using (\ref{seqn}), we find that
\be
 \rho_{ik} I^k{}_j+ \h R_{ij}-\frac{1}{4} P_{ij}=\h\nabla_i\theta_j\Longrightarrow \rho_{ik} I^k{}_j=\h\nabla_i (\theta_j+2\partial_j\Phi)=\h\nabla_i K_j~,
 \label{rhoeqn}
 \ee
 where
 \be
 \rho_{ij}=\frac{1}{2} \h R_{ijmn} \omega^{mn}~.
 \ee
 Thus
 \be
 \h\nabla_\alpha K_\beta+\h\nabla_\beta K_\alpha= \nabla_\alpha K_\beta+\nabla_\beta K_\alpha= i\rho_{\alpha\beta}+ i \rho_{\beta\alpha}=0
 \ee
 as $P$ is a $(1,1)$ tensor,
 where we have used complex coordinates $(x^i; i=1,\dots, 2n)=(z^\alpha, z^{\bar\alpha}; \alpha=1,\dots, n)$ with  $I= (i\delta^\alpha{}_{\beta}, -i \delta^{\bar\alpha}{}_{\bar\beta})$. This proves the first statement.

 For the proof of the second statement, note that in complex coordinates
 \be
 H_{\alpha\bar\beta\bar\gamma}=-\partial_{\bar\beta} g_{\alpha\bar\gamma}+\partial_{\bar\gamma} g_{\alpha\bar\beta}~,~~~\theta_\alpha= H_{\alpha\beta\bar\gamma} g^{\beta\bar\gamma}~.
 \ee
 Then a straightforward, though somewhat lengthy, calculation  reveals that
 \be
 D^i H_{i\beta\gamma}= g^{\alpha\bar\alpha} D_\alpha H_{\bar\alpha\beta\gamma}+g^{\alpha\bar\alpha}D_{\bar\alpha} H_{\alpha\beta\gamma}= -\partial_\beta \theta_\gamma+\partial_\gamma\theta_\beta- \theta_\delta H^\delta{}_{\beta\gamma}~,
 \ee
 where the integrability of the complex structure has been used, i.e. that $H_{\alpha\beta\gamma}=0$. Then one finds that
  \be
 \mathcal{L}_{\theta^\flat} \omega_{\beta\gamma}= i\partial_\beta \theta_\gamma- i\partial_\gamma\theta_\beta+ i \theta_\delta H^\delta{}_{\beta\gamma}~.
 \label{iiparti}
 \ee
 Thus, one concludes that
 \be
 D^iH_{i\beta\gamma}=i \mathcal{L}_{\theta^\flat} \omega_{\beta\gamma}~.
 \ee

 Next from (\ref{seqn}), we have   that
 \be
 D^i H_{i\beta\gamma}= 2 H^\delta{}_{\beta\gamma} \partial_\delta \Phi~.
 \ee
 and a computation reveals that
 \be
 (\mathcal{L}_{d\Phi^\flat} \omega)_{\beta\gamma}=i \partial_\delta\Phi H^\delta{}_{\beta\gamma}
 \label{iipartii}
 \ee
 Therefore,  from (\ref{iiparti}) and (\ref{iipartii}) we establish the second statement.

 Notice that in complex coordinates, $I= (i\delta^\alpha{}_{\beta}, -i \delta^{\bar\alpha}{}_{\bar\beta})$, the first and second properties imply that $\partial_{\bar\beta} K^\alpha=0$, i.e. that $K$ is holomorphic, and so
 \be
 \mathcal{L}_K I=0~.
 \ee

 The third property follows from (\ref{rhoeqn}) as
 \be
 \mathcal{L}_K g_{\alpha\bar\beta}=\nabla_\alpha K_{\bar\beta}+ \nabla_{\bar\beta} K_{\alpha}=\h\nabla_\alpha K_{\bar\beta}+ \h\nabla_{\bar\beta} K_{\alpha}=
 -i \rho_{\alpha\bar\beta}+ i \rho_{\bar\beta\alpha}= -2 i \rho_{\alpha\bar\beta}
 \ee
 and so
 \be
 \mathcal{L}_K \omega_{\alpha\bar\beta}=-i \mathcal{L}_K g_{\alpha\bar\beta}=-2 \rho_{\alpha\bar\beta}~.
 \ee

 Next as $K$ is holomorphic so is $Y=-IK$.  It remains to show that it is Killing. This follows from (\ref{rhoeqn}) as
 \be
 \rho_{ij}=-\h\nabla_i (K^\flat_m I^m{}_j)=-\h\nabla_i Y^\flat_j~.
 \ee
 Then using that $\rho$ is a 2-form, it is straightforward to see that $Y$ is Killing.
\end{proof}

\begin{remark}
As it has already been mentioned, for conformally balanced KT manifolds $\theta=-2d\Phi$ and so $K=0$. Thus, $Y=0$ as well. Therefore, not all KT manifolds that satisfy (\ref{seqn}) admit holomorphic isometries.
\end{remark}

\begin{prop}\label{prop:2}
Suppose that $(M^{2n}, g, \omega)$ is a compact KT manifold that satisfies  (\ref{seqn}) and $Z$ is a holomorphic and Killing vector field. Then, the dilaton field is invariant under the action of $Z$,  $Z\Phi=0$.  In particular, $Y$ leaves $\Phi$ invariant, $Y\Phi=0$, for $Y=-I K$.
\end{prop}
\begin{proof}
As $Z$ is holomorphic and Killing and $H$ is expressed in terms of the metric and complex structure, one has that $\mathcal{L}_Z H=0$. Similarly, one has that
$\mathcal{L}_Z P=0$ and $\mathcal{L}_Z \h R=0$. Then, using an argument as that deployed in the proof of Proposition \ref{prop:kinv}, we deduce that $Z\Phi=0$.

In Proposition \ref{prop:hol}, we have demonstrated that $Y=-IK$ is both holomorphic and Killing. Then, it follows from the argument above that $Y\Phi=0$ and so $\Phi$ is invariant under the action of $Y$.
\end{proof}

\subsubsection{Symmetries of CYT and HKT geometries with dilaton}

Suppose that $M^{2n}$ admits a CYT structure.  As a result the Ricci form vanishes, $\rho=0$. This leads to the following statement as a corollary of the Proposition \ref{prop:hol}.

\begin{corollary}\label{cor:CYT}
 Suppose that $(M^{2n}, g, \omega)$ is a CYT manifold that satisfies (\ref{seqn}), then $K$ and $Y=-IK$ with $K^\flat=\theta+2d\Phi$ are both holomorphic and $\h\nabla$-covariantly constant.  In particular, both $K$ and $Y$ are Killing vector fields.
\end{corollary}
\begin{proof}
It has already been demonstrated in proposition  \ref{prop:hol} that $K$ is holomorphic. Moreover, as $\rho=0$ and on comparing (\ref{seqn}) with (\ref{rhoeqn}) we deduce that
\be
\h\nabla K=0~,
\ee
i.e. that $K$ is $\h\nabla$-covariantly constant.  It follows from this that $K$ is Killing and $dK^\flat=\ii_K H$.

We have already demonstrated that $Y$ is holomorphic and Killing. It is also $\h\nabla$-covariantly constant as both $I$ and $K$ are $\h\nabla$-covariantly constant.
\end{proof}

\begin{remark}
As $K$ and $Y=-IK$ are holomorphic and $\h\nabla$-covariantly constant, it easily follows that they commute $[K,Y]=0$, see e.g. \cite{gp2}.  This is a consequence of the fact that the commutator of two $\h\nabla$-covariantly constant vector field can be expressed in terms of $\ii_K\ii_Y H$ and that $\ii_K H$ is a $(1,1)$-form as $K$ is holomorphic.  Furthermore, if $M^{2n}$ is compact, then as a consequence of Proposition \ref{prop:2} one has $K\Phi=Y\Phi=0$.
\end{remark}

Next let us turn to the HKT geometries $(M^{4k}, g, \omega_r)$.  These admit a hypercomplex structure $(I_r; r=1,2,3)$, with Hermitian forms $\omega_r(X,Z)=g(X, I_r Z)$, compatible with a connection $\h\nabla$ with skew-symmetric torsion $H$.  It follows that the holonomy of $\h\nabla$ is included in $Sp(k)$ and as a result all three Ricci forms
\be
(\rho_r)_{ij} =\frac{1}{2} \h R_{ikmn} I^k_{rj} \omega^{mn}_r~,~~~(\mathrm{no~summation~over~r})~,
\ee
vanish, $\rho_r=0$.  Another property of the HKT geometry, which is required for the proof of the corollary below, is that all three Lee forms $(\theta_r; r=1,2,3)$, with $\theta_r$ associated to the complex structure $I_r$,  are equal
\be
\theta=\theta_1=\theta_2=\theta_3~,
\ee
see e.g. \cite{pw} for a simple proof.
This is a property of the HKT geometry and it is valid irrespectively on whether $H$ is a closed 3-form or not.

\begin{corollary}\label{cor:hktsym}
Suppose that $(M^{4k}, g, \omega_r)$ is an HKT  manifold that satisfies (\ref{seqn}), then $K$ with $K^\flat=\theta+2d\Phi$ is tri-holomorphic and $\h\nabla$-covariantly constant vector field on $M^{4k}$, i.e. $K$ is holomorphic with respect to all three complex structures. Moreover, $Y_r=- I_r K$ is holomorphic with respect to the $I_r$ complex structure and $\h\nabla$-covariantly constant. In particular, the vector fields $K$ and $Y_r$ are Killing.
\end{corollary}
\begin{proof}
It is clear from the definition of HKT manifolds that they admit a CYT structure with respect to each one of the complex structures $I_r$. As the Lee form of the geometry does not depend on the choice of the complex structure that it is used to define it,  it follows from the previous corollary that $K$ must be holomorphic with respect to each one of the complex structures $I_r$ and $\h\nabla$-covariantly constant.  Thus, in particular, $K$ is tri-holomorphic.

It also follows from the previous corollary that $Y_r$ must be holomorphic with respect to the $I_r$ complex structure. It is also $\h\nabla$-covariantly constant as both $I_r$ and $K$ are $\h\nabla$-covariantly constant. The $\h\nabla$-covariant constancy condition implies that all $K$ and $Y_r$ are Killing vector fields.
\end{proof}
\begin{remark}
A similar argument to that deployed in the CYT case above implies that $K$ commutes with $Y_r$, i.e. $[K, Y_r]=0$.  However in general $[Y_r, Y_s]\not=0$, for $r\not= s$. An example of such symmetries is the Lie algebra generated by the left- or right-invariant vector fields of HKT geometry on $S^3\times S^1$ viewed as the group manifold $SU(2)\times U(1)$. In this example,  $K=\partial_\tau$ with $\tau$ the coordinate along $S^1$ and $Y_r=- I_r K$ are the left-invariant vector fields on $S^3$, where $I_r$ are the left-invariant complex structures on $S^3\times S^1$,  see \cite{pw} for a detailed analysis.  In more than four dimensions, $k>1$, a general structure theorem has recently been proved for the Lie algebra of the vector fields $Y_r$ in \cite{brienza}.  But this symmetry cannot automatically be lifted on $\mathscr{M}^*_{\mathrm {HE}}$, where the Hermitian-Einstein condition is taken with respect to one of the complex structures of the underlying HKT geometry, as the geometry of the $\mathscr{M}^*_{\mathrm {HE}}$ depends on the choice of complex structure and the vector fields $Y_r$ are not holomorphic with respect to $I_s$ for $r\not=s$.
\end{remark}

\begin{remark}
If in addition the HKT manifold $M^{4k}$ is compact, then as a consequence of Proposition \ref{prop:2}, one has that $K\Phi=Y_r\Phi=0$. In four dimensions, $M^4$ is parallelisable because $K$ and $Y_r$ are also $\h\nabla$-covariantly constant. Since $K$ commutes with $Y_r$ and the commutator is given by the components of $H$, $H$ can be written as $H=f \, Y^\flat_1\wedge Y^\flat_2\wedge Y^\flat_3$ for some function $f$. However, $H$ is invariant under the action of $K$ and $Y_r$, because $K$ is tri-holomorphic and Killing and $Y_r$ is holomorphic with respect to $I_r$ and Killing, which implies that $f$ is constant and $H$ is necessarily a closed 3-form. As $\Phi$ is invariant under the action of all four vectors fields $K$ and $Y_r$, $\Phi$ is constant. These manifolds for $H\not=0$  are identifications of $S^3\times S^1$ with a discrete group. Another way to state this result is that the unique set of hyper-complex structures on $S^3\times S^1$ \cite{boyer} do not admit HKT geometries with $dH\not=0$ and non-constant dilaton $\Phi$.  For a detailed description of the HKT structures on $S^3\times S^1$ see \cite{pw}.
\end{remark}

\subsection{Symmetries of heterotic geometries}

For generic geometries with skew-symmetric torsion\footnote{Examples of such geometries are completely reductive homogeneous spaces, where $\h\nabla$ is identified with the left invariant canonical connection \cite{OP}.}, (\ref{seqn}) does not arise from the extremisation of an (energy) functional via a Perelman type of argument. However, it does in some special cases, like for example whenever  $dH=0$ and $P=0$ \cite{OSW2, OSW} and in the context of heterotic perturbation theory  $dH, P\not=0$ \cite{gp2}. The derivation of (\ref{seqn}) using a variation principle is desirable because the (energy) functional can be used to utilise a Perelman type of argument as in \cite{GFS, pw} to specify the geometries that satisfy  (\ref{seqn}) by putting a restriction on the holonomy of $\h\nabla$. To be specific, the holonomy of $\h\nabla$ can be chosen such that the geometry satisfies a scale invariant type of condition, like that in (\ref{sinv}) or (\ref{sgfeqn}) below, and then the (energy) functional can be used to demonstrate that in fact the same geometry solves (\ref{seqn}).

In  heterotic string geometries $(M^d, g, H, C)$, one introduces another field $C$,  a gauge connection with curvature $F(C)$,  and (\ref{seqn}) is supplemented with the additional field equation for $C$. Therefore, the geometry satisfies the equations
\be
\h R_{ij}+ 2 \h\nabla_i\partial_j\Phi-{1\over 4} P_{ij}=0~,~~~e^{2\Phi} \h\nabla_A^i( e^{-2\Phi} F(C)_{ij})=0~.
\label{cgfeqn}
\ee
In addition, $P$  is given by
\be
P_{ij}=-\frac{\alpha'}{4} \big(\tilde R_{ijkm} \tilde R_j{}^{jkm}- F(C)_{ik \ma\mb} F_{j}{}^{k\ma\mb}\big)~.
\label{ph}
\ee
and moreover
\be
dH=\frac{\alpha'}{4}\big(\tilde R^i{}_j\wedge \tilde R^j{}_i- F(C)^\ma{}_\mb\wedge F(C)^\mb{}_\ma\big)~,
\label{dh}
\ee
with  $\tilde R$ the curvature of  a connection $\tilde\nabla$ of the tangent bundle of $M^{2n}$.  Heterotic geometries  are more restrictive than those with general $H$ and $P$. In perturbation theory with parameter $\alpha'$, $\tilde R$ is identified with the curvature $\breve R$ of the connection $\breve \nabla$ that has torsion $-H$. It is expected that both (\ref{cgfeqn}) and (\ref{dh}) receive higher order corrections in $\alpha'$.  The result stated above is valid up to the order in $\alpha'$ explicitly indicated.

The equations  (\ref{cgfeqn}) with $P$ and $dH$ given as above are the conditions for conformal invariance for the heterotic sigma model or equivalently gradient flow (generalised) solitons. While, the equations for scale invariance or equivalently for  steady flow (generalised)  solitons are
\be
\h R_{ij}-\frac{1}{4} P_{ij}=2\h\nabla_i X_j~,~~\h\nabla_A^i F(C)_{ij}+X^iF(C)_{ij}=0~,
\label{sgfeqn}
\ee
where $P$ and $dH$ are given in (\ref{ph}) and (\ref{dh}), respectively.

It has been established in \cite{gp2} in the context of heterotic perturbation theory that if $M^d$ is compact, then (\ref{sgfeqn}) implies (\ref{cgfeqn}). Therefore, (\ref{sgfeqn}) implies (\ref{cgfeqn}), either via a Perelman type of argument or otherwise, provided that
\be
\h\nabla K=0~,~~~\ii_K F(C)=0~,
\label{sccon3}
\ee
 where $K^\flat=X^\flat+2d\Phi$.

 \begin{remark}
 The integrability condition of $\h\nabla K=0$ implies that $\h R_{ij}{}^k{}_m K^m=0$. Working in the context of perturbation theory and so setting $\tilde R=\breve R$, see e.g. \cite{gp2} and reference therein for more explanation in a similar context, at the order considered in $\alpha'$, $\breve R_{ijkm}=\h R_{kmij}$, and so $\ii_K \breve R=0$. This together with (\ref{dh}) and the first and  second equations in (\ref{sccon3}) imply that $\mathcal{L}_K H=0$.  Therefore $K$ is $\h\nabla$-covariantly constant, and so Killing, and leaves $H$ invariant. Thus if $M^d$ is compact, it follows from the remark below Proposition \ref{prop:kinv} that  $K\Phi=0$. In turn such geometries satisfy
 the condition $D^i \big(e^{-2\Phi} K_i\big)=0$ that arises in the context of moduli spaces in the presence of a dilaton field.
 \end{remark}

 Next suppose that $(M^{2n}, g, \omega)$ is a CYT manifold. The compatibility of (\ref{dh}) with (\ref{ph}) via the Bianchi identity (\ref{bian2})  can be achieved the taking both $\tilde  R$ and $F(C)$ connections to be Hermitian-Einstein, i.e. to satisfy (\ref{HEcond1}),  with $\lambda=0$.  In perturbation theory for $\tilde R=\breve R$, this is automatically true as a consequence of the restriction of the holonomy  $\h \nabla$ to be included in $SU(n)$.
In such a case, the Bianchi identity of $F(C)$ implies that
\be
\h R_{ij}-\frac{1}{4} P_{ij}=2\h\nabla_i \theta_j~,~~\h \nabla_A^j F(C)_{ji}+\theta^j F(C)_{ji}=0~,
\ee
i.e. these geometries satisfy (\ref{sgfeqn}) for $X^\flat=\theta$.
Therefore, either a Perelman style of argument or simply imposing that these geometries should also satisfy (\ref{cgfeqn})  implies
\be
\h\nabla K=0~,~~~\ii_K F(C)=0~,
\ee
with $K^\flat=\theta+2 d\Phi$.
It is also clear that $\ii_Y F(C)=0$ for $Y=-IK$ as $F(C)$ is a $(1,1)$-form. Adapting the results of Corollary \ref{cor:CYT} to this case, we can demonstrate the following proposition.

\begin{prop}\label{prop:3}
Suppose that $(M^{2n}, g, \omega)$ is a CYT manifold with $dH$ given in (\ref{dh}), where $F(C)$ is a Hermitian-Einstein connection with $\lambda=0$. Assuming that  $(M^{2n}, g, \omega)$ is either compact or simply  satisfies (\ref{cgfeqn}),  $M^{2n}$ admits two commuting, holomorphic and $\h\nabla$-covariantly constant  vector fields $K$ and $Y=-IK$ with  $K^\flat=\theta+ 2d\Phi$ provided that $K^\flat\not=0$.
\end{prop}

\begin{remark}
 At first order in perturbation theory, $H$ is a closed 3-form,  $dH=0$. In the context of  supersymmetric string compactifications with internal spaces CYT manifolds $M^{2n}$, $M^{2n}$ is conformally balanced $\theta=-2 d\Phi$, i.e. $K=0$,  as a consequence of the dilatino Killing spinor equation up to at least second order in perturbation theory.  Then, it follows from a theorem in \cite{SIGP3} that $M^{2n}$ is a Calabi-Yau manifold with $H=0$ at first order in perturbation theory. As Calabi-Yau manifolds do not admit Killing vector fields, it is likely that $K=0$ in all orders in perturbation theory. Therefore, one has to consider non-supersymmetric compactifications, like compactifications with  internal spaces  group manifolds, in order for $K$ not to vanish.
\end{remark}

\begin{remark}
The statement of Proposition \ref{prop:3} can be adapted   to HKT manifolds $(M^{4k}, g, \omega_r)$,  as they are special cases of CYT manifolds, provided that $F(C)$ is a $(1,1)$-form with respect to all three complex structures $I_r$. The difference is that $K$ is tri-holomorphic and $\h\nabla$-covariantly constant while $Y_r=-I_r K$ are holomorphic with respect to $I_r$ and $\h\nabla$-covariantly constant.
\end{remark}

\subsection{Application to Hermitian-Einstein connection moduli spaces}

It is clear that Hermitian geometries that satisfy (\ref{seqn}) with $K^\flat=\theta+2d\Phi$  admit a number of holomorphic and Killing vector fields.
Then, it is a consequence of the results of section \ref{sec:d} that the moduli space of Hermitian-Einstein connection will admit holomorphic and Killing vector field as well.  The results can be summarised in the theorem below.

\begin{theorem}\label{th:modkill}
Let $(M^{2n}, g,\omega)$ be a compact KT manifold that satisfies (\ref{seqn}) and $K$, with  $K^\flat=\theta+2d\Phi\not=0$, is a Killing vector field.  Then, $\mathscr{M}_\HE^*(M^{2n})$ admits a strong KT structure and  two  holomorphic and Killing vector field $\alpha_K$ and  $\alpha_Y$ for $Y=-IK$.  In addition, if $(M^{2n}, g,\omega)$ is a CYT manifold, then $\mathscr{M}_\HE^*(M^{2n})$ admits two commuting holomorphic and $\h{\mathcal{D}}$-covariantly constant vector field $\alpha_K$ and $\alpha_Y$.
\end{theorem}
\begin{proof}
If $(M^{2n}, g,\omega)$ is a compact KT manifold and $K$ is Killing, then the condition $D^i \big(e^{-2\Phi} K^\flat_i\big)=0$, see  (\ref{ktheta}),  required to describe the geometry of $\mathscr{M}_\HE^*(M^{2n})$ is satisfied as a consequence of the Proposition \ref{prop:2}.  Therefore, it follows from Theorem \ref{th:skt} that  $\mathscr{M}_\HE^*(M^{2n})$ has a strong KT structure.

From the hypothesis of the theorem $K$ is Killing. In addition, Proposition \ref{prop:hol} implies that $K$ is also holomorphic, and $Y=-IK$ is both holomorphic and Killing. As a result the associated vector fields on $\alpha_K$ and $\alpha_Y$ on the moduli space $\mathscr{M}_\HE^*(M^{2n})$ are holomorphic and Killing -- this follows from Theorem \ref{th:holkill}.

If $(M^{2n}, g,\omega)$ is a compact CYT manifold that satisfies (\ref{seqn}), then $K$ and $Y$ are holomorphic and $\h\nabla$-covariantly constant. As it has already been demonstrated that $\alpha_K$ and $\alpha_Y$ are holomorphic and Killing because CYT manifolds are a special case of KT manifolds, it remains to demonstrate that  $\alpha_K$ and $\alpha_Y$ are $\h{\mathcal{D}}$-covariantly constant.  This follows from  Theorem \ref{th:phihnabla} as both $K^\flat\wedge K^\flat=0$ and $Y^\flat\wedge K^\flat$ are clearly $(1,1)$-forms. Moreover, $\alpha_K$ and $\alpha_Y$ commute because $K$ and $Y$ commute.  The latter follows  because both are $\h\nabla$-covariantly constant and $K$ is holomorphic, i.e. $i_KH$ is a $(1,1)$-form. As a result their commutator,  given by $\ii_K\ii_Y H$, vanishes. \end{proof}

The above theorem admits various refinements provided a Perelman type of argument can be applied. This is especially the case provided that $(M^{2n}, g,\omega)$
is a CYT manifolds. In particular, one can show the following.

\begin{corollary}
Let $(M^{2n}, g,\omega)$ be a compact CYT manifold with either $dH=P=0$ or $P$ and $dH$ given in (\ref{ph}) and  (\ref{dh}), respectively, where $F(C)$ is a Hermitian-Einstein connection with $\lambda=0$ in the context of heterotic geometries. Then, $\mathscr{M}_\HE^*(M^{2n})$ admits two commuting holomorphic and $\h{\mathcal{D}}$-covariantly constant vector field $\alpha_K$ and $\alpha_Y$ provided that $K\not=0$.
\end{corollary}
\begin{proof}
It suffices to demonstrate that such a geometry satisfies (\ref{seqn}). Then the result follows from theorem (\ref{th:modkill}).  This is indeed the case as a consequence of a Perelman type of argument as described in \cite{GFS, pw, gp2}.
\end{proof}

\begin{remark}

The Hermitian-Einstein condition (\ref{HEcond1}) with $\lambda=0$ on connections of bundles over 4-dimensional manifold implies that the connections are anti-self-dual with respect to the orientation on $M^4$ given by the complex structure $I$, i.e. the volume 4-form is chosen as $1/2\, \omega\wedge\omega$, where $\omega$ is the Hermitian form.  In fact, the anti-self-duality condition is independent from the choice of complex structure on $M^4$. This is the case provided that any two choices of a complex structure induce the same orientation on $M^4$.  The main consequence of this is that the moduli space of anti-self-dual connections can exhibit a more intricate geometric structure induced by judicious choices of a complex structure on $M^4$.  Such structures have been explored in \cite{hitchin, moraru, witten, gp}.  Moreover, there is an additional coincidence in 4-dimensions because the closure of the 3-form torsion $H$ of a KT geometry on $M^4$ implies the Gauduchon condition and vice versa.  Furthermore, the CYT and HKT structures coincide in four dimensions. In such a case for $M^4$ compact and for $H\not=0$, a consequence of the second remark below Corollary \ref{cor:hktsym} implies that $M^4$ must be identified with $S^3\times S^1$ up to an identification with a discrete group. The properties of the moduli spaces of anti-self-dual connections on $S^3\times S^1$ have been extensively been explored in \cite{witten}, see also \cite{gp}.

It remains to consider the geometry of instanton moduli spaces for $(M^4, g, \omega)$ a compact KT manifold. This has already been dealt with in  Theorem \ref{th:modkill}. It should also be noted that if $M^4$ admits an oriented bi-KT structure, or equivalently oriented generalised K\"ahler \cite{hitchin, gualtieri}, then the analogue of the condition (\ref{ktheta})  associated to second KT structure  with connection $\breve \nabla$ that has torsion $-H$ is
\be
D^i \big(e^{-2\Phi} (-\theta_i+2\partial_i \Phi)\big)=0~.
\label{kbtheta}
\ee
This is because the Lee forms $\h\theta$ and $\breve\theta$ of the two KT structures that give the bi-KT structure on $M^4$ must satisfy  $\theta\equiv \h\theta=-\breve\theta$ for the associated complex structures $\h I$ and $\breve I$  to induce the same orientation on $M^4$, see \cite{hitchin, witten}.
Comparing (\ref{kbtheta}) with (\ref{ktheta}), we conclude that $\Phi$ is constant and $\theta$ is divergence free, i.e. $g$ is in the Gauduchon gauge.  In such a case, the geometry and symmetries
of the moduli space of instantons $\mathscr{M}^*_\asd$ have been already explored in \cite{hitchin, witten} and  \cite{gp}, respectively.
 \end{remark}

\vskip1cm
 \noindent {\it {Acknowledgements}}

 I thank Edward Witten for inspiration to work  on the geometry of moduli spaces and Jeff Streets for correspondence.

 \bibliographystyle{unsrt}

\end{document}